\documentclass[12pt]{amsart}

\newtheorem{theorem}{Theorem}[section]
\newtheorem{lemma}[theorem]{Lemma}

\newtheorem{proposition}[theorem]{Proposition}

\newtheorem{corollary}[theorem]{Corollary}

\usepackage[bookmarks]{hyperref}
\usepackage{fullpage}
\usepackage{cite}
\usepackage{cleveref,babel,graphicx,amscd,amsthm,amsfonts,amsopn,amssymb,verbatim,enumerate,xcolor}

\usepackage[letterpaper,left=1in,top=1in,right=1in,bottom=1in]{geometry}

\def\logd{\log |D_x|}

\def\tv{\tilde v}
\def\tw{\tilde w}
\def\half{ \frac{1}{2}}
\def\D{\partial}

\def\R{{\mathbb R}}

\def\nint{\mathop{\diagup\kern-13.0pt\int}}
\def\Z{{\mathbb Z}}

\def\bas{\begin{align*}}
\def\eas{\end{align*}}
\def\bi{\begin{itemize}}
\def\ei{\end{itemize}}

\def\emph#1{{\it #1}}

\def\eps{{\epsilon}}

\def\dq{{\delta}}
\def\sdq{{|\delta|}}

\def\pax{\mathbf{u}}
\def\rpax{\mathbf{r}}

\DeclareMathOperator{\sgn}{sgn}

\theoremstyle{definition}
\newtheorem{remark}[theorem]{Remark}
\numberwithin{equation}{section}

\begin{document}

\title{Low regularity solutions for the surface quasi-geostrophic front equation}

\author{Albert Ai}
\address{Department of Mathematics, University of Wisconsin, Madison}
\email{aai@math.wisc.edu}

\author{Ovidiu-Neculai Avadanei}
\address{Department of Mathematics, University of California at Berkeley}
\email{ovidiu\_avadanei@berkeley.edu}

\begin{abstract}
In this article we consider the low regularity well-posedness of the surface quasi-geostrophic (SQG) front equation. Recent work on other quasilinear models, including the gravity water waves system and nonlinear waves, have demonstrated that in presence of a null structure, a normal form analysis can substantially improve the low regularity theory. In the current article, we observe a null structure in the context of SQG fronts, and establish improved local and global well-posedness results.
\end{abstract}


\maketitle
\addtocontents{toc}{\protect\setcounter{tocdepth}{1}}
\tableofcontents

\section{Introduction}

The surface quasi-geostrophic (SQG) equation takes the form 
\begin{equation}\label{rawSQG}
\theta_t +  u \cdot \nabla \theta = 0, \qquad u = (- \Delta )^{-\half} \nabla^\perp \theta
\end{equation}
where $\theta$ is a scalar evolution on $\R^2$, $(-\Delta)^{-\half}$ denotes a fractional Laplacian, and $\nabla^\perp = (-\D_y, \D_x)$. The SQG equation arises from oceanic and atmospheric science as a model for quasi-geostrophic flows confined to a surface. This equation is also of interest due to similarities with the three dimensional incompressible Euler equation. In particular, the question of singularity formation remains open for both problems.

The SQG equation is one member in a family of two-dimensional active scalar equations parameterized by the transport term in \eqref{rawSQG}, with 
\begin{equation}\label{gSQG}
u = (- \Delta )^{-\frac{\alpha}{2}} \nabla^\perp \theta, \qquad \alpha \in (0, 2].
\end{equation}
The case $\alpha = 2$ gives the two dimensional incompressible Euler equation, while the $\alpha = 1$ case gives the SQG equation \eqref{rawSQG} above.

\

Front solutions to \eqref{rawSQG} refer to piecewise constant solutions taking the form
\[
\theta(t, x, y) = \begin{cases}
  \theta_+  \qquad \text{if } y > \varphi(t, x), \\
  \theta_-  \qquad \text{if } y < \varphi(t, x),
\end{cases}
\]
where the front is modeled by the graph $y = \varphi(t, x)$ with $x \in \R$. Front solutions are closely related to patch solutions
\[
\theta(t, x, y) = \begin{cases}
  \theta_+  \qquad \text{if } (x, y) \in \Omega(t), \\
  \theta_-  \qquad \text{if } (x, y) \notin \Omega(t),
\end{cases}
\]
where $\Omega$ is a bounded, simply connected domain. For instance, when $\alpha \in (0, 1)$, the derivation and analysis of contour dynamics equations do not differ substantially between the case of patches and the case of fronts. Global well-posedness for the front equation for small and localized data was established by Córdoba-Gómez-Serrano-Ionescu in \cite{GlobalPatch}. 

However, when $\alpha \in [1, 2]$, the derivation of contour dynamics equations for fronts has additional complexities relative to the case of patches, arising from the slow decay of Green's functions. The derivation in this range was provided by Hunter-Shu \cite{HSderivation} via a regularization procedure, and again by Hunter-Shu-Zhang in \cite{HSZderivation}. In the SQG case $\alpha = 1$, the evolution equation for the front $\varphi$ takes the form 
\begin{equation}\label{SQG}
\begin{aligned}
(\partial_t - 2\log|D_x|\D_x)\varphi(t,x) &= Q(\varphi , \D_x\varphi)(t, x), \\
\varphi(0,x)&=\varphi_0(x),
\end{aligned}
\end{equation}
where $\varphi$ is a real-valued function $\varphi : [0, \infty) \times \R \rightarrow \R$ and
\begin{equation}\label{A-op}
Q(f,g)(x) = \int \left(\frac{1}{|y|} - \frac{1}{\sqrt{y^2+(f(x+y)-f(x))^2}}\right)\cdot (g(x + y) - g(x)) \,dy.
\end{equation}
This can be rewritten using difference quotients as
\begin{equation}\label{A-op2}
Q(f,g)(x) =\int F(\dq^yf) \cdot \sdq^yg\,dy,
\end{equation}
where
\[
\displaystyle F(s)=1-\frac{1}{\sqrt{1+s^2}},\qquad \displaystyle \dq^yf(x)=\frac{f(x+y)-f(x)}{y}, \qquad \displaystyle \sdq^yg(x)=\frac{g(x+y)-g(x)}{|y|}.
\]

The equation \eqref{SQG} is invariant under the transformation
\begin{align*}
    t\rightarrow\kappa t,\qquad x\rightarrow\kappa(x+2\log|\kappa|t),\qquad\varphi\rightarrow\kappa\varphi,
\end{align*}
which implies that $\dot{H}^{\frac{3}{2}}(\mathbb{R})$ is the corresponding critical Sobolev space.

\

In the case of SQG patches, Gancedo-Nguyen-Patel proved in \cite{GNP} that under a suitable parametrization, the contour dynamics evolution is locally well-posed in $H^s(\mathbb{T})$ when $s>2$. Local well-posedness for the generalized SQG family in the case $\alpha\in(0,2)$ and $\alpha \neq 1$ was also considered by Gancedo-Patel in \cite{GP}, establishing in particular local well-posedness in $H^2$ for $\alpha \in (0, 1)$ and in $H^3$ for $\alpha \in (1, 2)$. For a more recent result on enhanced lifespan for $\alpha$-patches, see Berti-Cuccagna-Gancedo-Scrobogna \cite{BCGS}.

Concerning ill-posedness for patches, Kiselev-Luo \cite{KL} proved results in Sobolev spaces with exponents $p \neq 2$, as well as in H\"older spaces. Further, Zlato\v s \cite{Zlatos} proved finite time blow up for patch solutions (both bounded and unbounded) for suitable initial data and $\alpha\in(0,1/4]$, provided that local well-posedness for \eqref{rawSQG} is known in the Sobolev space $H^3(\mathbb{R}^2)$ (corresponding to solutions in $H^2(\mathbb{R})$ in the contour dynamics formulation from \cite{GP}).

In the current article we are concerned with SQG fronts. Work in this area began with Hunter-Shu-Zhang, who studied the local well-posedness for a cubic approximation of \eqref{SQG} in \cite{HSZapprox}, and subsequently in \cite{HSZglobal} established local well-posedness for the full equation \eqref{SQG} with initial data in $H^s$, $s \geq 5$, along with global well-posedness for small, localized, and essentially smooth ($s \geq 1200$) initial data. These results were extended to the range $\alpha \in (1, 2)$ for the generalized SQG family in \cite{HSZfamily}, while also improving the required regularity to $s > \frac72 + \frac{3\alpha}{2}$. However, these well-posedness results require a small data assumption to ensure the coercivity of the modified energies used in the energy estimates, along with a convergence condition on an expansion of the nonlinearity $Q(\varphi,\varphi_x)$ appearing in \eqref{SQG}.

In \cite{SQGzero}, the authors lowered the regularity thresholds for both the local and global theory, proving that the problem is locally well-posed in $H^s$ when $s>\frac{5}{2}$, which corresponds to the classical energy threshold of Hughes-Kato-Marsden \cite{HKM} at one derivative above scaling; and globally well-posed for small and localized initial data in $H^s$ when $s>4$.

\

In the present article, our objective is to revisit and streamline the analysis of \eqref{SQG}, while improving the established well-posedness results. Our contributions include:
\begin{itemize}
\item establishing the local well-posedness in a significantly lower regularity setting at $\frac{1}{2}+\epsilon$ derivatives above scaling, by observing a null structure for the equation and carrying out an associated normal form analysis, and
\item establishing the global well-posedness in a low regularity setting, by applying the wave packet testing method of Ifrim-Tataru (see for instance \cite{ITschrodinger, ITpax}).
\end{itemize}
We anticipate that our streamlined analysis will open the way to substantial simplifications and improvements in the analysis of related equations, including the generalized SQG family \eqref{gSQG}.

\subsection{Main results}

We recall for the purpose of comparison the energy estimate in \cite{SQGzero} for \eqref{SQG},
\begin{equation}\label{energy0}
\frac{d}{dt}E^{(s)}(\varphi)\lesssim_A AB_0 \cdot E^{(s)}(\varphi),
\end{equation}
where
\[
A = \|\D_x \varphi\|_{L^\infty}, \qquad B_0 = \|\D_x \varphi\|_{C^{1, \delta}}.
\]
We remark that this energy estimate is cubic, which is consistent with the fact that the nonlinearity $Q$ of \eqref{SQG} is fully nonlinear, and cubic at leading order. 

However, from the perspective of the balance of derivatives, \eqref{energy0} manifests quadratically, in that a full extra derivative is absorbed solely by the $B_0$ control norm. This in turn can be understood by viewing the cubic nonlinearity $Q$ as a quadratic nonlinearity with a low frequency coefficient, and motivates our choice of notation $Q = Q(f, g)$.

A key observation in the current article is that the SQG equation \eqref{SQG} satisfies a resonance structure, in the sense that, using the quadratic perspective, the nonlinearity can be viewed as 
\begin{equation}
    Q(f,g)\approx\Omega(\partial_x^{-1}F(f_x),g),
\end{equation}
where the symbol of $\Omega$ is
\begin{equation*}
\Omega(\xi_1, \xi_2)=\omega(\xi_1) + \omega(\xi_2) - \omega(\xi_1 + \xi_2), \qquad \omega(\xi) = 2i \xi \log|\xi|.
\end{equation*}
For more details, see the discussion in Section \ref{s:equations}.

As a result, we have access to normal form methods which open the way to a refined energy estimate with a better balance of derivatives,
\[
\frac{d}{dt}E^{(s)}(\varphi)\lesssim_A B^2 \cdot E^{(s)}(\varphi),
\]
where here we have balanced the control norm
\[
 B := \|\D_x \varphi\|_{C^{\half +}}.
\]

The gain obtained from the rebalanced energy estimate, along with its counterpart for the linearized equation, leads to our main local well-posedness result:
\begin{theorem}\label{t:lwp}
Equation \eqref{SQG} is locally well-posed for initial data in $\dot H^{s_0} \cap \dot H^s$ with $s > 2$ and $s_0<\frac32$.  Precisely, for every $R > 0$, there exists $T=T(R)>0$ such that for any $\varphi_0\in \dot H^{s_0} \cap \dot H^s$ with $\|\varphi_0\|_{\dot H^{s_0} \cap \dot H^s} < R$, the Cauchy problem \eqref{SQG} has a unique solution $\varphi \in C([0, T], \dot H^{s_0} \cap \dot H^s)$. Moreover,  the solution map $\varphi_0 \mapsto \varphi$ from $\dot H^{s_0} \cap \dot H^s$ to $C([0, T], \dot H^{s_0} \cap \dot H^s)$ is continuous.
\end{theorem}
\begin{remark}
    In order to keep the proofs simpler, we assume that the parameter $A$ is small. This assumption can be removed with a more careful definition of the paraproduct, at the expense of having to deal with more technical details.
\end{remark}

Normal forms were first introduced by Shatah in \cite{Shatah} to study the long-time behavior of solutions to dispersive equations. However, in the quasilinear context, the normal form transformation is not readily applicable, because the resulting correction will not be bounded. Several approaches have been introduced to address this, and in the present article we primarily rely on two. First is the idea of carrying out the normal form analysis in a paradifferential fashion, which was first used by Alazard-Delort \cite{Alazard-Delort} in a paradiagonalization argument used to obtain Sobolev estimates for the solutions of the water waves equations in the Zakharov formulation. This approach was also used by Ifrim-Tataru \cite{Benjamin} to obtain a new proof of the $L^2$ global well-posedness for Benjamin-Ono equation, proved previously by Ionescu-Kenig \cite{Ionescu-Kenig}. 

Second is the use of modified energies, in lieu of the direct normal form transform at the level of the equation. This approach was first introduced by Hunter-Ifrim-Tataru-Wong \cite{Hunter-Ifrim-Tataru-Wong} to study long time solutions of the Burgers-Hilbert equation.

The combination of these approaches to address the low regularity theory for quasilinear models was first introduced by the first author with Ifrim-Tataru in \cite{Ai-Ifrim-Tataru} for the gravity water waves system, through the proof of \emph{balanced energy estimates}. Balanced energy estimates were subsequently further combined with Strichartz estimates in the context of the low regularity theory for the time-like minimal surface problem in the Minkowski space \cite{Minimal-surface}.  

\

We remark that our local well-posedness threshold of $s > 2$ coincides with the result of Gancedo-Nguyen-Patel \cite{GNP} for patches. However, the use of the null structure in our current work yields stronger energy estimates in the sense that our control norms $A$ and $B$ consist of only pointwise norms rather than $L^2$-based norms. This allows for further applications, including the analysis of long time behavior, which we discuss next.

We will consider global well-posedness for small and localized data. To describe localized solutions, we define the operator 
\[
L = x+2t+2t\log|D_x|,
\]
which commutes with the linear flow $\partial_t-2\log|D_x|\D_x$, and at time $t = 0$ is simply multiplication by $x$. Then we define the time-dependent weighted energy space
\[
\|\varphi\|_X := \|\varphi\|_{\dot{H}^{s_0}\cap\dot{H}^s} + \|L\partial_x \varphi\|_{L^2},
\]
where $s>3$ and $s_0<1$. To track the dispersive decay of solutions, we define the pointwise control norm 
\[
\|\varphi\|_Y := \||D_x|^{1-\delta}\langle D_x\rangle^{\frac{1}{2}+2\delta}\varphi\|_{L_x^{\infty}}.
\]

\begin{theorem}\label{t:gwp}
Consider data $\varphi_0$ with
\[
\|\varphi_0\|_X \lesssim \eps \ll 1.
\]
Then the solution $\varphi$ to \eqref{SQG} with initial data $\varphi_0$ exists globally in time, with energy bounds 
\[
\|\varphi(t)\|_{X} \lesssim \eps t^{C\eps^2}
\]
and pointwise bounds 
\[
\|\varphi(t)\|_Y \lesssim \eps \langle t\rangle^{-\half}.
\]
\end{theorem}
 Further, the solution  $\varphi$ exhibits a modified scattering behavior, with an asymptotic profile $W\in H^{1-C_1\epsilon^2}(\R)$, in a sense that will be made precise in Section~\ref{s:scattering}.

\

\subsection{Outline of the paper}

Our paper is organized as follows. In Section~\ref{s:notation}, we establish notation and preliminaries used through the rest of the paper, including estimates involving the paradifferential calculus and difference quotients.

In Section~\ref{s:equations}, we introduce the null structure of equation \eqref{SQG} and its linearization,
\begin{equation}\label{linearized-eqn}
    \partial_tv- 2\log|D_x|\partial_xv = \partial_xQ(\varphi,v).
\end{equation}
We also introduce the paradifferential flow associated to \eqref{linearized-eqn}, which will play a central role in the subsequent analysis.

In Section~\ref{s:reduction}, we reduce the energy estimates and well-posedness of the linearized equation \eqref{linearized-eqn} to that of the inhomogeneous paradifferential flow. The primary difficulty is ensuring that the paradifferential errors satisfy balanced cubic estimates. In order to achieve this, we carry out a paradifferential normal form analysis to remove the unbalanced components of the errors.

In Section~\ref{s:energy estimates}, we establish energy estimates for the paradifferential flow. Here, we construct a quasilinear modified energy, which reproduces the one introduced by Hunter-Shu-Zhang in \cite{HSZglobal}. However, in the current article, we carefully observe cancellations which ensure that our estimates are balanced.

In Section~\ref{s:higher order energy estimates}, we establish higher order energy estimates. Extra care must be taken because the commutators are quadratic rather than cubic, and thus not perturbative. In order to eliminate them, we use an exponential Jacobian conjugation combined with a normal form correction.

In Section~\ref{s:lwp}, we prove Theorem~\ref{t:lwp}, the local well-posedness result for \eqref{SQG}. We use the method of frequency envelopes to construct rough solutions as the unique limit of smooth solutions. This method was introduced by Tao in \cite{25} to better track the evolution of energy distribution between dyadic frequencies. A systematic presentation of the use of frequency envelopes in the study of local well-posedness theory for quasilinear problems can be found in the expository paper \cite{ITprimer}. 

In Section~\ref{s:gwp} we use the wave packet testing method of Ifrim-Tataru to prove the global-wellposedness part of Theorem \ref{t:gwp}, along with the dispersive bounds for the resulting solution. This method was systematically presented in \cite{ITpax}. Finally, in Section~\ref{s:scattering} we discuss the modified scattering behavior of the global solutions constructed in Section \ref{s:gwp}.

\subsection{Acknowledgements}

The first author was supported by the NSF grant DMS-2220519 and the RTG in Analysis and Partial Differential equations grant DMS-2037851. The second author was supported by the NSF
grant DMS-2054975, as well as by the Simons Foundation. 

The authors were also supported by the NSF under Grant No. DMS-1928930 while in residence at the Simons Laufer Mathematical Sciences Institute (formerly MSRI) in Berkeley, California, during the summer of 2023.

The authors would like to thank Mihaela Ifrim and Daniel Tataru for many helpful discussions.

\section{Notations and classical estimates}\label{s:notation}

In this section we discuss some notation and classical estimates that we use throughout the article. These include estimates involving the paradifferential calculus, and difference quotients.

\subsection{Paradifferential operators and paraproducts}

Let $\chi$ be an even smooth function such that $\chi=1$ on $[-\frac{1}{20}, \frac{1}{20}]$ and $\chi = 0$ outside $[-\frac{1}{10}, \frac{1}{10}]$, and define
\[
\tilde{\chi}(\theta_1,\theta_2) = \chi\left(\frac{|\theta_1|^2}{M^2+|\theta_2|^2}\right).
\]
Given a symbol $a(x,\eta)$,  we use the above cutoff symbol $\tilde \chi$ to define an $M$ dependent paradifferential quantization of $a$ by (see also \cite{ABZgravity}) 
\begin{align*}
    \widehat{T_au}(\xi)=(2\pi)^{-1}\int \hat{P}_{>M}(\xi)\tilde \chi\left(\xi - \eta, \xi + \eta \right) \hat{a}(\xi-\eta,\eta)\hat{P}_{>M}(\eta)\hat{u}(\eta)\,d\eta,
\end{align*}
where the Fourier transform of the symbol $a=a(x,\eta)$ is taken with respect to the first argument.

This quantization was employed in \cite{SQGzero}, where the parameter $M$ was introduced to ensure the coercivity of the modified quasilinear energy without relying on a small data assumption. We recall in particular that in the case of a paraproduct, where $a = a(x)$ is real-valued, $T_a$ is self-adjoint.

\

The following commutator-type estimates are exact reproductions of statements from Lemmas 2.4 and 2.6 in Section 2 of \cite{Ai-Ifrim-Tataru}, respectively:

\begin{lemma}[Para-commutators]\label{l:para-com}
 Assume that $\gamma_1, \gamma_2 < 1$. Then we have
\begin{equation}\label{para-com}
\| T_f T_g - T_g T_f \|_{\dot H^{s} \to \dot H^{s+\gamma_1+\gamma_2}} \lesssim 
\||D|^{\gamma_1}f \|_{BMO}\||D|^{\gamma_2}g\|_{BMO},
\end{equation}
\begin{equation}
\| T_f T_g - T_g T_f \|_{\dot B^{s}_{\infty,\infty} \to \dot H^{s+\gamma_1+\gamma_2}} \lesssim 
\||D|^{\gamma_1}f \|_{L^2}\||D|^{\gamma_2}g\|_{BMO}.
\end{equation}
A bound similar to \eqref{para-com} holds in the Besov scale of spaces, namely 
from $\dot B^{s}_{p, q}$ to $\dot B^{s+\gamma_1+\gamma_2}_{p, q}$
for real $s$ and $1\leq p,q \leq \infty$.
\end{lemma}

The next paraproduct estimate, see Lemma 2.5 in \cite{Ai-Ifrim-Tataru}, directly relates multiplication and paramultiplication:

\begin{lemma}[Para-products]\label{l:para-prod}
Assume that $\gamma_1, \gamma_2 < 1$, $\gamma_1+\gamma_2 \geq 0$. Then
\begin{equation}
\| T_f T_g - T_{fg} \|_{\dot H^{s} \to \dot H^{s+\gamma_1+\gamma_2}} \lesssim 
\||D|^{\gamma_1}f \|_{BMO}\||D|^{\gamma_2}g\|_{BMO}.
\end{equation}
A similar bound holds in the Besov scale of spaces, namely 
from $\dot B^{s}_{p, q}$ to $\dot B^{s+\gamma_1+\gamma_2}_{p, q}$
for real $s$ and $1\leq p,q \leq \infty$.
\end{lemma}

Next, we recall the following Moser-type estimate. See for instance \cite{SQGzero}.
\begin{theorem}[Moser]\label{t:moser}
Let $F:\R \rightarrow \R$ be a smooth function with $F(0) = 0$, and 
\[
R(v) = F(v) - T_{F'(v)} v.
\]
Then 
\begin{equation}\label{moser}
\|R(v)\|_{W^{\half, \infty}}\lesssim_{\|v\|_{L^\infty}} \||D|^\half v\|_{L^\infty}.
\end{equation}
\end{theorem}

\subsection{Difference quotients}

We recall that we denote difference quotients by 
\[
\dq^yh(x)=\frac{h(x+y)-h(x)}{y}, \qquad \sdq^yh(x)=\frac{h(x+y)-h(x)}{| y|},
\]
as well as the smooth function 
\[
F(s) = 1 - \frac{1}{\sqrt{1+s^2}},
\]
which in particular vanishes to second order at $s = 0$, satisfying $F(0)=F'(0)=0$. Using this notation, we may express
\[
    Q(\varphi, v)(t,x) = \int F(\dq^y\varphi(t,x)) \cdot \sdq^y v(t,x)\,dy.
\]
In addition, to facilitate the normal form analysis in later sections, we denote
\[
\psi := \D_x^{-1} F(\varphi_x).
\]

We have the following estimate which allows the balancing of up to one derivative over multilinear averages of difference quotients:

\begin{lemma}\label{Trilinear integral estimate}
Let $i = \overline{1,n}$ and $p_i, r \in [1, \infty]$ and $\alpha_i, \beta_i \in [0, 1]$ satisfying
\[
\sum_{i}\frac{1}{p_i} = \frac{1}{r}, \qquad n-1 < \sum_{i}\alpha_i\leq n, \qquad 0 \leq \sum_{i}\beta_i < n-1.
\]
Then
\[
 \left\|\int \prod \dq^yf_i \,dy\right\|_{L_x^r} \lesssim \prod \||D|^{\alpha_i} f_i\|_{L^{p_i}} + \prod \||D|^{\beta_i} f_i\|_{L^{p_i}}.
 \]

\end{lemma}

\begin{proof}
We write 
\[
\int \prod \dq^yf_i \,dy = \int_{|y| \leq 1} + \int_{|y| > 1}.
\]
For the former integral, we have by H\"older
\[
 \left\| \int_{|y| \leq 1} \prod \dq^yf_i \,dy \right\|_{L_x^r} \lesssim \int_{|y| \leq 1} \frac{1}{|y|^{n - \sum \alpha_i}} \prod \||D|^{\alpha_i} f_i\|_{L^{p_i}} \, dy \lesssim \prod \||D|^{\alpha_i} f_i\|_{L^{p_i}}.
\]
The latter integral is treated similarly.
\end{proof}

\section{The null structure and paradifferential equation}\label{s:equations}

In this section and the next, we will reduce the energy estimates and well-posedness for the linearized equation \eqref{linearized-eqn}, 
\[
\D_tv - 2\logd \D_xv = \D_x Q(\varphi, v),
\]
to that of a paradifferential equation.

One can achieve this by viewing \eqref{linearized-eqn} as a paradifferential equation with perturbative source, where the main task is to paralinearize the cubic term $\D_x Q(\varphi, v)$. Such an analysis was performed by Hunter-Shu-Zhang in \cite{HSZglobal} and refined by the authors in \cite{SQGzero}.

In the current article however, we are interested in further refining the paralinearization of \eqref{linearized-eqn} by insisting that the perturbative errors satisfy \emph{balanced estimates}. Precisely, we establish all of our estimates using only the control norms
\[
A := \|\D_x \varphi\|_{L^\infty}, \qquad B := \|\D_x \varphi\|_{C^{\half +}},
\]
where $A$ corresponds to the scaling-critical threshold, while $B$ lies half a derivative above scaling. For comparison, the local well-posedness analysis in \cite{SQGzero} uses control norms with $\varphi \in C^{2+}$, a full derivative above scaling.

A direct estimate of the paralinearization errors will no longer suffice to establish estimates controlled by $A$ and $B$. Instead, we will rely on an appropriate paradifferential normal form transformation to remove source components that do not directly satisfy the desired balanced cubic estimates. In this section, we first consider various formulations of the paradifferential equation which will be useful in the following sections.

\subsection{Null structure}

Although $F(\dq^y \varphi)$ is principally quadratic in $\varphi$ (and thus $Q(\varphi, v)$ is cubic), estimates on derivatives of $F(\dq^y \varphi)$ do not fully recognize its quadratic structure. This is because they are limited by the cases of low-high interaction where derivatives fall on the high frequency variable. As a result, from the perspective of establishing balanced estimates, $F(\dq^y \varphi)$ behaves essentially like a linear coefficient. 

However, we observe that $Q$ exhibits a null structure in the following sense. By writing
\[
\Omega(f, g) = \int \dq^yf \cdot \sdq^yg\,dy
\]
and using the heuristic approximation
\[
F(\dq^y \varphi) \approx T_{F'(\varphi_x)} \dq^y \varphi,
\]
we may express $Q$ as a quadratic form with low frequency coefficient,
\begin{equation}\label{atoq}
Q(\varphi, v) \approx T_{F'(\varphi_x)}\Omega(\varphi, v).
\end{equation}
We then observe that the bilinear form $\Omega(\varphi, v)$ exhibits a null structure, since its symbol
\[
\Omega(\xi_1, \xi_2) = \int \sgn{y} \cdot \frac{(e^{i\xi_1 y}-1)(e^{i\xi_2 y}-1)}{y^2} \, dy
\]
satisfies the resonance identity
\begin{equation}\label{resonance}
\Omega(\xi_1, \xi_2) = \omega(\xi_1) + \omega(\xi_2) - \omega(\xi_1 + \xi_2), \qquad \omega(\xi) = 2i \xi \log|\xi|.
\end{equation}
This null structure underlies the normal form analysis, which we perform in the next section.

\

We make the above discussion precise in the following lemma. Recall that we denote
\[
\psi := \D_x^{-1} F(\varphi_x).
\]

\begin{lemma}\label{l:toy-red}
We have
\[
Q(\varphi, v) = \Omega(\psi, v) + R(x, D) v
\]
where 
\[
\|(\D_x R)(x, D) v\|_{L^2} \lesssim_A B^2 \|v\|_{L^2}.
\]
\end{lemma}

\begin{proof}
We write
\begin{equation}
\begin{aligned}
Q(\varphi, v) - \Omega(\psi, v) &= \int \frac{F(\dq^y \varphi) - \dq^y \D_x^{-1} F(\varphi_x)}{|y|} \cdot (v(x + y) - v(x)) \, dy =: R(x, D)v
\end{aligned}
\end{equation}
where
\[
r(x, \xi) = \int \frac{F(\dq^y \varphi) - \dq^y \D_x^{-1} F(\varphi_x)}{|y|} (e^{i\xi y} - 1) \, dy.
\]
Then we have
\begin{equation*}
\begin{aligned}
(\D_x R)(x, D)v &= \int \frac{F'(\dq^y \varphi)\dq^y \varphi_x - \dq^y F(\varphi_x)}{|y|} \cdot (v(x + y) - v(x)) \, dy \\
& =: \int K(x, y) \cdot (v(x + y) - v(x)) \, dy.
\end{aligned}
\end{equation*}

We first estimate $K$, which we may write as 
\begin{equation*}
\begin{aligned}
|y| K(x, y) &= \frac{1}{y} (F'(b)(a - b) - (F(a) - F(b))+\frac{1}{y} (F'(\dq^y \varphi)-F'(\varphi_x))(a - b) \\
&=: |y| K_1(x, y) + |y| K_2(x, y),
\end{aligned}
\end{equation*}
where $a = \varphi_x(x + y)$, $b = \varphi_x(x)$. From $K_1$ we obtain a Taylor expansion,
\[
\|K_1(\cdot, y)\|_{L^\infty_x} \lesssim_A  \left\|\frac{a - b}{y} \right\|_{L^\infty_x}^2 = \|\dq^y \varphi_x\|_{L^\infty_x}^2.
\]
For $K_2$, we have 
\begin{align*}
\|K_2(\cdot, y) \|_{L^\infty_x} &\lesssim_A \left\|\frac{\varphi(x+y)-\varphi(x)-y\varphi_x(x)}{y^2} \right\|_{L^\infty_x} \left\|\frac{a - b}{y} \right\|_{L^\infty_x} =  \|\dq^{y, (2)} \varphi\|_{L^\infty_x} \|\dq^y \varphi_x\|_{L^\infty_x},
\end{align*}
where $\dq^{y, (2)}$ denotes the second-order difference quotient.

By Minkowski's inequality,
\begin{equation*}
\begin{aligned}
\|(\D_x R)(x, D)v \|_{L^2} &\lesssim \int \|K(\cdot, y)\|_{L^\infty_x} \|v(\cdot + y) - v(\cdot) \|_{L^2_x} \, dy \\
&\lesssim_A  \int\|\dq^y \varphi_x\|_{L^\infty_x} (\|\dq^y \varphi_x\|_{L^\infty_x} + \|\dq^{y, (2)} \varphi\|_{L^\infty_x} )\|v \|_{L^2_x} \, dy.
\end{aligned}
\end{equation*}
Applying the argument of Lemma~\ref{Trilinear integral estimate} with $\alpha_1 = \alpha_2 = \half+$ and $\beta_1 = \beta_2 = \half-$, we conclude
\[
\|(\D_x R)(x, D)v \|_{L^2} \lesssim_A B^2 \|v\|_{L^2}.
\]

\

\end{proof}

\subsection{The paradifferential flow}

Associated to the linearized equation \eqref{linearized-eqn}, we have the corresponding inhomogeneous paradifferential flow,
\begin{equation}\label{paradiff-eqn}
\D_tv - 2\logd \D_xv - \D_x Q_{lh}(\varphi, v) = f,
\end{equation}
where we have expressed the frequency decomposition of the (essentially) quadratic form as
\begin{equation}\label{paralin}
\begin{aligned}
Q(\varphi, v) &= \int T_{F(\dq^y \varphi)} \sdq^yv\,dy + \int T_{\sdq^yv} F(\dq^y \varphi) \,dy + \int \Pi(\sdq^yv, F(\dq^y \varphi)) \,dy \\
&=: Q_{lh}(\varphi, v) + Q_{hl}(\varphi, v) + Q_{hh}(\varphi, v).
\end{aligned}
\end{equation}
We frequency decompose $\Omega = \Omega_{lh} + \Omega_{hl} + \Omega_{hh}$ in the analogous way.

In Section~\ref{s:reduction}, we show that the linearized equation \eqref{linearized-eqn} reduces to the paradifferential flow \eqref{paradiff-eqn}, with $f$ playing a perturbative role, in the sense that it satisfies balanced, cubic estimates. However, since $Q$ and hence its paradifferential errors $Q_{hl}(\varphi, v)$ and $Q_{hh}(\varphi, v)$ are essentially quadratic, this will become apparent only after an appropriate paradifferential normal form change of variables. 

Here, we establish a preliminary quadratic estimate for the reduction, which will be useful in the course of constructing and evaluating the normal form transformation later in Section~\ref{s:reduction}.

\

We first extract the principal components of $\Omega_{lh}$, which include a transport term of order 0 and a dispersive term of logarithmic order:

\begin{lemma}\label{l:q-principal}
We can express
\begin{equation}\label{qlh-expression}
\Omega_{lh}(\psi, v) = 2 (T_{\logd \D_x \psi} v - T_{\D_x\psi} \logd v + \D_x [T_\psi, \logd] v).
\end{equation}
Further, we have
\begin{equation}\label{qlh-expression-full}
Q_{lh}(\varphi, v) = 2(T_{(\logd - 1) \D_x \psi + R} v - T_{\D_x \psi} \logd v) + \Gamma(\D_x^2 \psi, \D_x^{-1} v)
\end{equation}
where
\begin{equation}\label{l-est}
\||D_x|^\half \Gamma\|_{L^2} \lesssim_A B \|v\|_{L^2}, \qquad \|\D_x \Gamma\|_{L^2} \lesssim_A B \||D|^\half v\|_{L^2},
\end{equation}
as well the pointwise estimates
\begin{equation}\label{l-est-infty}
\begin{aligned}
&\||D_x|^\half \Gamma\|_{L^\infty} \lesssim_A B \|v\|_{L^\infty}, \qquad  &\|\D_x \Gamma\|_{L^\infty} \lesssim_A B \||D|^\half v\|_{L^\infty}.
\end{aligned}
\end{equation}
\end{lemma}

\begin{proof}
We use the resonance identity \eqref{resonance} to expand
\begin{equation}
\begin{aligned}
\Omega_{lh}(\psi, v) &= 2 T_{\logd \D_x \psi} v + 2[T_\psi, \logd \D_x]v \\
&=  2 T_{\logd \D_x \psi} v - 2 T_{\D_x\psi} \logd v + 2\D_x [T_\psi, \logd] v,
\end{aligned}
\end{equation}
obtaining \eqref{qlh-expression}. Then the remaining commutator may be expressed as
\[
-2T_{\D_x \psi} v + \Gamma(\D_x^2 \psi, \D_x^{-1} v)
\]
and $\Gamma$ denotes the subprincipal remainder, which has a favorable balance of derivatives on the low frequency and thus may be estimated as \eqref{l-est} and \eqref{l-est-infty}. Combined with the low-high component of Lemma~\ref{l:toy-red}, we obtain \eqref{qlh-expression-full}.
\end{proof}

\begin{proposition}\label{p:quadratic-lin}
Consider a solution $v$ to \eqref{linearized-eqn}. Then $v$ satisfies
\begin{equation}\label{paradiff-eqn-2}
(\D_t - 2T_{1 - \D_x\psi} \logd \D_x - 2T_{(\logd - 1) \D_x\psi + R} \D_x) v = f
\end{equation}
where
\begin{equation}\label{quadratic-est}
\||D|^{-\half}f\|_{L^2} \lesssim_A B \|v\|_{L^2}.
\end{equation}
\end{proposition}

\begin{proof}
We express \eqref{linearized-eqn} in terms of the paradifferential equation \eqref{paradiff-eqn} with source,
\[
\D_tv - 2 \logd \D_xv - \D_x Q_{lh}(\varphi, v) = \D_x Q_{hl}(\varphi, v) + \D_x Q_{hh}(\varphi, v).
\]

We estimate the source terms. We see directly from definition that $\D_x Q_{hl}(\varphi, v)$ has a favorable balance of derivatives which satisfies \eqref{quadratic-est} and may be absorbed into $f$. The balanced $Q_{hh}$ term similarly satisfies \eqref{quadratic-est}, so we have thus reduced \eqref{linearized-eqn} to \eqref{paradiff-eqn}.

It then suffices to apply \eqref{qlh-expression-full} of Lemma~\ref{l:q-principal} to the remaining paradifferential $Q_{lh}$ term on the left hand side of \eqref{paradiff-eqn} to obtain \eqref{paradiff-eqn-2}. Here, the $\Gamma$ contribution may be absorbed into $f$ directly. Further, we have commuted the $\D_x$ outside $Q_{lh}$ through the low frequency paracoefficients, since the cases where this derivative falls on the low frequency coefficients,
\[
 2(T_{(\logd - 1) \D_x^2 \psi + \D_x R} v - T_{\D_x^2 \psi} \logd v),
\]
have a favorable balance of derivatives, satisfying \eqref{quadratic-est}.

\end{proof}

\subsection{Nonlinear equations}

We will also use the paradifferential equation \eqref{paradiff-eqn-2} in the context of the nonlinear solutions $\varphi$. To conclude this section, we establish preliminary quadratic bounds on the inhomogenity of the paradifferential flow, in analogy with the preceding Proposition~\ref{p:quadratic-lin} for the linearized counterpart.

\begin{proposition}\label{p:psi-eqn}
 Consider a solution $\varphi$ to \eqref{SQG}. Then $\varphi$ satisfies
\begin{equation}\label{psi-eqn}
(\D_t - 2T_{1 - \D_x \psi} \logd \D_x - 2T_{(\logd - 1) \D_x \psi + R} \D_x) \varphi = f
\end{equation}
where
\begin{equation}\label{psi-eqn-est}
\||D_x|^{\half} f\|_{L^\infty} \lesssim_A B, \qquad \|\D_x f\|_{L^\infty} \lesssim_A B^2.
\end{equation}
The same holds for $\psi$ in the place of $\varphi$.
\end{proposition}

\begin{proof}

We first consider the case of $\varphi$. We paradifferentially decompose $Q(\varphi, \D_x \varphi)$ in \eqref{SQG} to write it in terms of the paradifferential equation \eqref{paradiff-eqn} with source,
\[
\D_t\varphi - 2 \logd \D_x \varphi - Q_{lh}(\varphi, \D_x \varphi) = Q_{hl}(\varphi, \D_x \varphi) + Q_{hh}(\varphi, \D_x \varphi).
\]

As with the linearized equation, we estimate the source terms. We see directly from definition that $Q_{hl}(\varphi, \D_x \varphi)$ has a favorable balance of derivatives which satisfies \eqref{psi-eqn-est} and may be absorbed into $f$. The balanced $Q_{hh}$ term similarly satisfies \eqref{psi-eqn-est}, so we have thus replaced $Q(\varphi, \D_x \varphi)$ in \eqref{SQG} with $Q_{lh}(\varphi, \D_x \varphi)$. In turn, it then suffices to apply Lemma~\ref{l:q-principal} to obtain \eqref{psi-eqn}.

\

We next reduce the equation for $T_{F'(\varphi_x)} \varphi$ to that of $\varphi$. It suffices to apply the paracoefficient $T_{F'(\varphi_x)}$ to \eqref{psi-eqn}, and estimate the commutators. This is straightforward for the spatial paradifferential terms, applying Lemma~\ref{l:para-com} and observing a favorable balance of derivatives.

For the time derivative, we substitute \eqref{psi-eqn} for the time derivative of $\varphi$:
\[
T_{F''(\varphi_x) \D_x \D_t\varphi} \varphi = T_{F''(\varphi_x) (2\D_x T_{1 - \D_x \psi} \logd \D_x\varphi + 2\D_x T_{(\logd - 1) \D_x \psi + R} \D_x \varphi + \D_x f)} \varphi.
\]
The estimate \eqref{psi-eqn-est} on $\D_x f$ in the paracoefficient implies that its contribution in this context also satisfies \eqref{psi-eqn-est}. For the remaining terms, the favorable balance of derivatives, with two or more derivatives on the low frequency paracoefficient, again implies that we may absorb their contribution into $f$. 

\

To conclude the proof for $\psi$, it suffices to apply the Moser estimate of Theorem~\ref{t:moser}, other than for the time derivative, for which we need to estimate 
\[
\D_x^{-1} (F'(\varphi_x) \D_x \D_t\varphi) - \D_t T_{F'(\varphi_x)} \varphi.
\]
We decompose this into
\[
[\D_x^{-1}, T_{F'(\varphi_x)}] \D_x \D_t\varphi
\]
which we estimate directly, using the favorable balance of derivatives, and 
\[
\D_x^{-1} T_{\D_x \D_t\varphi} F'(\varphi_x) + \D_x^{-1}\Pi(\D_x \D_t\varphi, F'(\varphi_x))
\]
which is similar to the time derivative commutation in the previous reduction.

\end{proof}

\section{Reduction to the paradifferential equation}\label{s:reduction}

Our objective in this section is to reduce the energy estimates and well-posedness of the linearized equation \eqref{linearized-eqn} to that of the inhomogeneous paradifferential equation \eqref{paradiff-eqn},
\[
\D_tv - 2\logd \D_xv - \D_x Q_{lh}(\varphi, v) = f.
\]
To ensure that the energy estimates are balanced, we require that the inhomogeneity $f$ satisfies balanced cubic estimates.

However, the paradifferential errors $Q_{hl}(\varphi, v)$ and $Q_{hh}(\varphi, v)$ are essentially quadratic rather than cubic, and in particular do not satisfy balanced cubic estimates. On the other hand, we saw in \eqref{atoq} that up to leading order and a low frequency coefficient, $Q$ is the quadratic form associated to the resonance function for the dispersion relation of \eqref{SQG}. This motivates the normal form change of variables
\begin{equation}\label{full-nf}
\tilde v = v - \D_x (\psi v), \qquad \psi = \D_x^{-1} F(\varphi_x).
\end{equation}

However, \eqref{full-nf} suffers from two shortcomings:
\begin{enumerate}
\item We cannot make use of \eqref{full-nf} directly, as it is an unbounded transformation, and
\item quartic (essentially cubic) residuals in the equation for $\tv$ given by \eqref{full-nf} are still unbalanced.
\end{enumerate}
To address the first shortcoming, we instead consider a bounded paradifferential component of \eqref{full-nf},
\begin{equation}\label{para-nf}
\tilde v = v - \D_x (T_v \psi) - \D_x \Pi(v, \psi)
\end{equation}
which is compatible with our objective of reducing to the paradifferential equation \eqref{paradiff-eqn}. To address the second shortcoming, we refine \eqref{para-nf} with a low frequency Jacobian coefficient which addresses the quartic and higher order residuals:
\begin{equation}\label{j-para-nf}
\tilde v = v - \D_x T_{T_J v} \psi - \D_x \Pi(T_J v, \psi), \qquad J = (1 - \D_x \psi)^{-1}.
\end{equation}

\begin{proposition}\label{Linearized normalized variable}
Consider a solution $v$ to \eqref{linearized-eqn}. Then we have
\begin{equation}\label{paradiff-eqn-inhomog}
\D_t\tv - 2\logd \D_x\tv - \D_x Q_{lh}(\varphi, \tv) = \tilde f,
\end{equation}
where $\tilde f$ satisfies balanced cubic estimates,
\begin{equation}\label{bal-est}
\|\tilde f\|_{L^2} \lesssim_A B^2 \|v\|_{L^2}.
\end{equation}
\end{proposition}

\begin{proof}

We express $v$ satisfying \eqref{linearized-eqn} in terms of the paradifferential equation \eqref{paradiff-eqn} with source,
\[
\D_tv - 2 \logd \D_xv - \D_x Q_{lh}(\varphi, v) = \D_x Q_{hl}(\varphi, v) + \D_x Q_{hh}(\varphi, v).
\]
Unlike in Proposition~\ref{p:quadratic-lin}, we do not estimate the source terms directly. Instead, we will establish the following cancellation with the contribution from the normal form correction,
\begin{equation}\label{main-cancellation}
\D_t \D_x T_{T_J v} \psi - 2 \logd \D_x^2 T_{T_J v} \psi - \D_x Q_{lh}(\psi, \D_x T_{T_J v} \psi) = \D_x Q_{hl}(\varphi, v) + \tilde f,
\end{equation}
with the analogous relationship for the balanced $\Pi$ component of the correction, with $Q_{hh}$.

To show \eqref{main-cancellation}, we first observe that using Lemma~\ref{l:q-principal}, we may replace $\D_x Q_{lh}$ on the left hand side of \eqref{main-cancellation} by its principal components. The $\Gamma$ error is estimated using the second estimate of \eqref{l-est},
\[
\|\D_x \Gamma\|_{L^2} \lesssim_A B \||D|^\half \D_x T_{T_J v} \psi\|_{L^2} \lesssim_A B^2 \|v\|_{L^2}
\]
and may be absorbed into $\tilde f$. It thus suffices to show
\begin{equation}\label{lh-simplified}
(\D_t - 2\D_x (T_{1 - \D_x \psi} \logd - T_{(\logd - 1)\D_x \psi + R})) \D_x T_{T_J v} \psi = \D_x Q_{hl}(\varphi, v) + \tilde f.
\end{equation}

\

Next, we compute the time derivative in \eqref{lh-simplified}. The case where $\D_t$ falls on the low frequency $J$ may be absorbed into $\tilde f$ due to a favorable balance of derivatives. More precisely, we use \eqref{psi-eqn} to write
\[
\D_t J = J^2 \D_x \D_t \psi = J^2\D_x(2T_{1 - \D_x \psi} \logd \D_x \psi + 2T_{(\logd - 1) \D_x \psi + R} \D_x \psi + f)
\]
so that we can we can estimate for instance the contribution of the first term,
\[
\|\D_x T_{T_{2 J^2 \D_x^2 \logd \psi} v} \psi \|_{L^2} \lesssim_A B^2 \|v\|_{L^2},
\]
with similar estimates for the other contributions, using the first estimate of \eqref{psi-eqn-est} for the contribution of $f$.

In the remaining cases, $\D_t$ falls on the middle frequency $v$ or the high frequency $\psi$, so we use \eqref{paradiff-eqn-2} and \eqref{psi-eqn} respectively to write
\begin{equation}\label{dt-nf}
\begin{aligned}
\D_x T_{T_J \D_t v} \psi &= \D_x T_{T_J(2T_{1 - \D_x \psi} \logd \D_xv + 2T_{(\logd - 1) \D_x \psi + R} \D_x v + f)} \psi, \\
\D_x T_{T_J v} \D_t \psi &= \D_x T_{T_J v} (2T_{1 - \D_x \psi} \logd \D_x\psi + 2T_{(\logd - 1) \D_x \psi + R} \D_x \psi + f_\psi).
\end{aligned}
\end{equation}
We consider the first, second, and third contributions from the two equations of \eqref{dt-nf} in pairs:

\

i) The first terms in \eqref{dt-nf} combine with the second term on the left in \eqref{lh-simplified},
\begin{equation}\label{pre-q}
2\D_x (T_{T_J T_{1 - \D_x \psi} \logd \D_xv} \psi + T_{T_J v} T_{1 - \D_x \psi} \logd \D_x\psi - T_{1 - \D_x \psi} \logd \D_x T_{T_J v} \psi),
\end{equation}
to form $\D_x Q_{hl}(\varphi, v)$, modulo balanced errors which may be absorbed into $\tilde f$. To see this, we will use in each of the three terms that $(1 - \D_x \psi)J = 1$. As we do so, we need to take care that any paraproduct errors and commutators yield balanced errors. 

First, we observe that in the third term, we can apply the commutator estimate
\[
\|\D_x [T_{1 - \D_x \psi}, \logd \D_x] T_{T_J v} \psi \|_{L^2} \lesssim_A B^2 \|v\|_{L^2}.
\]
Then using Lemma~\ref{l:para-prod} to compose paraproducts in each of the three terms of \eqref{pre-q}, we have
\[
2\D_x (T_{\logd \D_xv} \psi + T_{(1 - \D_x \psi) T_J v} \logd \D_x\psi -  \logd \D_x T_{(1 - \D_x \psi)T_J v} \psi).
\]

For the latter two terms, we will also use Lemma~\ref{l:para-prod} to compose paraproducts, before applying $(1 - \D_x \psi)J = 1$. To do so, we first need to exchange multiplication by $(1 - \D_x \psi)$ with the paraproduct $T_{1 - \D_x \psi}$. However, the error from this exchange is not directly perturbative. Instead, we perform the exchange for the two terms simultaneously, to observe a cancellation in the form of the commutator
\[
2\D_x(T_{T_{T_J v} (1 - \D_x \psi)} \logd \D_x\psi -  \logd \D_x T_{T_{T_J v}(1 - \D_x \psi)} \psi),
\]
which has a favorable balance of derivatives and may be absorbed into $\tilde f$. The same holds for the analogous cases with $\Pi(1 - \D_x \psi, T_J v)$. We have thus reduced \eqref{pre-q} to 
\[
2\D_x (T_{\logd \D_xv} \psi + T_{v} \logd \D_x\psi -  \logd \D_x T_{v} \psi) = \D_x \Omega_{hl}(\psi, v)
\]
which by Lemma~\ref{l:toy-red} coincides with $\D_x Q_{hl}(\varphi, v)$ up to balanced errors, as desired.

\

ii) The second terms in \eqref{dt-nf},
\begin{equation}\label{second-term}
2\D_x (T_{T_J T_{(\logd - 1) \D_x \psi + R} \D_x v} \psi + T_{T_J v} T_{(\logd - 1) \D_x \psi + R} \D_x \psi),
\end{equation}
combine to cancel the third term on the left hand side of \eqref{lh-simplified}, up to balanced errors. To see this, we apply the commutator Lemma~\ref{l:para-com} to exchange the first term of \eqref{second-term} with
\[
2\D_x T_{T_{(\logd - 1) \D_x \psi + R} T_J \D_x v} \psi.
\]
We can freely exchange the low frequency paraproduct $T_{(\logd - 1) \D_x\psi + R}$ with a standard product, since
\begin{equation}\label{favorable-bal}
\|\D_x T_{T_{T_J \D_x v} (\logd - 1) \D_x \psi + T_{T_J \D_x v} R} \psi\|_{L^2} \lesssim_A B^2 \|v\|_{L^2}
\end{equation}
and likewise for the balanced $\Pi$ case. We thus have 
\[
2\D_x T_{T_J \D_x v \cdot (\logd - 1) \D_x \psi + R} \psi.
\]
Then applying Lemma~\ref{l:para-prod} for splitting paraproducts, and returning to \eqref{second-term}, we arrive at
\[
2\D_x (T_{T_J \D_x v} T_{(\logd - 1) \D_x \psi + R} \psi + T_{T_J v} T_{(\logd - 1) \D_x \psi + R} \D_x \psi).
\]
Lastly, we factor out a derivative,
\[
2\D_x^2 T_{T_J v} T_{(\logd - 1) \D_x \psi + R} \psi
\]
where we have absorbed the cases where the factored derivative falls on $J$ or $(\logd - 1) \D_x \psi + R$ into $\tilde f$, similar to \eqref{favorable-bal}. After one more instance of the commutator Lemma~\ref{l:para-com}, we arrive at the third term on the left hand side of \eqref{lh-simplified} as desired.

\

iii) By Propositions~\ref{p:quadratic-lin} and \ref{p:psi-eqn} respectively, the contributions from $f$ and $f_\psi$ satisfy \eqref{bal-est} and may be absorbed into $\tilde f$.

\end{proof}

We also obtain a similar but easier balanced estimate for the reduction of the nonlinear equation to the paradifferential flow, in the $\dot H^s$ setting. Here the normal form correction consists only of a balanced $\Pi$ component:

\begin{equation}\label{j-para-nf-2}
\tilde{\varphi} = \varphi -  \Pi(\psi,T_J \D_x\varphi).
\end{equation}

\begin{proposition}\label{Normalized variable}
Consider a solution $\varphi$ to \eqref{SQG}. Then we have
\begin{equation}\label{paradiff-eqn-inhomog-2}
\D_t\tilde{\varphi} - 2\logd \D_x\tilde{\varphi} -  \D_x Q_{lh}(\varphi, \tilde{\varphi}) = \tilde f,
\end{equation}
where $\tilde f$ satisfies balanced cubic estimates,
\begin{equation}\label{bal-est-2}
\|\tilde f\|_{\dot{H}^s} \lesssim_A B^2 \|\varphi\|_{\dot{H}_x^s}.
\end{equation}
\end{proposition}

\begin{proof}

First observe that we have
\[
\D_t\varphi - 2\logd \D_x \varphi -  \D_x Q_{lh}(\varphi, \varphi) = Q_{hh}(\varphi, \D_x \varphi).
\]
Then the normal form analysis is similar to the analysis for the (balanced) paradifferential errors of the linear equation in Proposition~\ref{Linearized normalized variable}.

To see that we can obtain balanced estimates in the $\dot H^s$ setting, first observe that in each of the estimates in the proof of Proposition~\ref{Linearized normalized variable}, we can easily obtain at least one $B$ from the estimate of the low frequency variable. Then since we are in the balanced $\Pi$ setting, we can obtain a second $B$, with $s$ outstanding derivatives, which can then be placed on the remaining high frequency factor.

\end{proof}

\section{Energy estimates for the paradifferential equation}\label{s:energy estimates}

In this section we establish energy estimates for the paradifferential equation \eqref{paradiff-eqn}. We define the modified energy
\[
E(v) := \int v \cdot T_{1 - \D_x \psi} v \, dx.
\]

\begin{proposition}\label{Paradiferential flow linearized energy estimates}
We have
\begin{equation}\label{bal-energy}
\frac{d}{dt} E(v) \lesssim_A B^2 \|v\|_{L^2}^2 + \|f\|_{L^2} \|v\|_{L^2}.
\end{equation}

\end{proposition}

\begin{proof}
Without loss of generality we assume $f = 0$. We consider first the linear component of the energy. Using the equation \eqref{paradiff-eqn} for $v$ and skew adjointness of $\logd \D_x$, we have
\[
\half \frac{d}{dt} \int v \cdot v \, dx = \int (2\logd \D_xv + \D_x Q_{lh}(\varphi, v)) \cdot v \, dx = \int \D_x Q_{lh}(\varphi, v) \cdot v \, dx.
\]
Using Lemma~\ref{l:q-principal} to expand $Q_{lh}$, this may be written
\[
\int 2\D_x(T_{\logd \D_x  \psi + R} v + T_{\psi}\logd \D_x  v - \logd \D_x T_\psi v) \cdot v \, dx.
\]
Cyclically integrating by parts (the first term individually, and the latter two terms paired), this may be expressed as
\begin{equation}\label{cubic-terms}
\int (T_{\D_x^2 \logd  \psi + \D_x R} v + 2T_{\D_x \psi}\logd \D_x  v) \cdot v \, dx.
\end{equation}
Here, the contribution with $\D_x R$ is balanced by Lemma~\ref{l:toy-red} and may be discarded.

\

Next we evaluate the effect of the quasilinear modification to the energy, using \eqref{paradiff-eqn} and \eqref{psi-eqn} respectively to expand time derivatives:
\begin{equation}
\begin{aligned}
\half \frac{d}{dt} \int v \cdot T_{\D_x \psi} v \, dx &= \int \D_t v \cdot T_{\D_x \psi} v \, dx +\half \int v \cdot T_{\D_x \D_t\psi} v \, dx \\
&= \int (2\logd \D_xv + \D_x Q_{lh}(\varphi, v)) \cdot T_{\D_x \psi} v \, dx \\
&\quad + \half \int v \cdot T_{\D_x (2T_{1 - \D_x \psi} \logd \D_x \psi + 2T_{(\logd - 1) \D_x \psi + R} \D_x \psi + f)} v \, dx.
\end{aligned}
\end{equation}
The contribution from $f$ may be estimated using \eqref{psi-eqn-est} and discarded. The cubic terms cancel with \eqref{cubic-terms}, so it remains to estimate the quartic terms,
\begin{equation}\label{quartic-terms}
\int \D_x Q_{lh}(\varphi, v) \cdot T_{\D_x \psi} v + v \cdot T_{\D_x (T_{(\logd - 1) \D_x \psi + R} \D_x\psi - T_{\D_x \psi} \logd \D_x \psi)} v \, dx.
\end{equation}

We expand $Q_{lh}$ in the first term of \eqref{quartic-terms}, and will observe cancellations in two steps with the remaining terms. The expansion is similar to the expansion of the $L^2$ energy above, except with an additional $T_{\D_x \psi}$ paraproduct:
\begin{equation}\label{quartic-terms-2}
\int 2\D_x (T_{\logd \D_x \psi + R} v + T_{\psi}\logd \D_x  v - \logd \D_x T_\psi v) \cdot T_{\D_x \psi} v \, dx.
\end{equation}

i) From the first term in \eqref{quartic-terms-2}, we have after cyclically integrating by parts, 
\[
\int v \cdot T_{\D_x^2 \logd \psi + \D_x R}T_{\D_x \psi} v - v \cdot T_{\logd \D_x  \psi + R}T_{\D_x^2 \psi} v + v \cdot [T_{\D_x \psi}, T_{\logd \D_x  \psi + R}]\D_x v \, dx.
\]
The commutator satisfies \eqref{bal-energy} by Lemma~\ref{l:para-com}, up to errors also satisfying \eqref{bal-energy}, as well as the non-perturbative residual
\begin{equation}\label{residual-q}
-\int v \cdot T_{\D_x T_{\D_x \psi} \D_x\psi} v \, dx
\end{equation}
which we will address in the next step. To see these cancellations, we use Lemma~\ref{l:para-prod} to compose paraproducts, and observe that any contribution with two or more derivatives on the lowest frequency has a favorable balance of derivatives and satisfies \eqref{bal-energy}. For instance, from the second term of \eqref{quartic-terms}, we have the perturbative component
\[
\left|\int v \cdot T_{T_{\D_x^2\logd \psi} \D_x\psi} v \, dx \right| \lesssim_A B^2 \|v\|_{L^2}.
\]

ii) It remains to estimate the last two terms in \eqref{quartic-terms-2}, along with \eqref{residual-q}. In the last term of \eqref{quartic-terms-2}, the case where the $\D_x$ falls on the low frequency $\psi$ vanishes by skew adjointness. From what remains, we have the commutator
\[
-2 \int [T_\psi, \logd] \D_x v \cdot \D_x T_{\D_x \psi} v \, dx = 2 \int T_{\D_x \psi} \D_x [T_\psi, \logd] \D_x v \cdot v \, dx.
\]
 On the other hand, due to skew adjointness, we are free to insert
\[
\int 2 T_{\D_x \psi} \D_x T_{\D_x \psi} v \cdot v \, dx = 0
\]
which subtracts the principal component of the commutator. In addition, we can rewrite \eqref{residual-q}, up to perturbative errors, as
\[
-\int T_{\D_x \psi} T_{\D_x^2\psi} v \cdot v \, dx.
\]
Combined, we have
\begin{equation}\label{l-commutator}
\int T_{\D_x \psi} L(\D_x^2 \psi, v) \cdot v \, dx
\end{equation}
where
\[
L(\D_x^2 \psi, v) = (2\D_x [T_\psi, \logd] \D_x + 2\D_x T_{\D_x \psi} - T_{\D_x^2 \psi}) v.
\]
Since the components 
\[
2\D_x [T_\psi, \logd] \D_x, \qquad 2\D_x T_{\D_x \psi} - T_{\D_x^2 \psi}
\]
of $L(\D_x^2 \psi, \cdot)$ are both skew-adjoint, $L(\D_x^2 \psi, \cdot)$ is skew-adjoint as well. We thus have a commutator form for \eqref{l-commutator},
\[
\int T_{\D_x \psi} L(\D_x^2 \psi, v) \cdot v \, dx = \half \int (T_{\D_x \psi} L(\D_x^2 \psi, v) -  L(\D_x^2 \psi, T_{\D_x \psi} v)) \cdot v \, dx
\]
for which we have the desired balanced estimate.

\end{proof}
By combining this result with the normal form analysis from the previous section, we obtain the following well-posedness result:
\begin{proposition}\label{Linearized equation energy estimates}
Assume that $A\ll 1$ and $B\in L_t^2$. There exists an energy functional $E_{lin}(v)$ such that for every solution of \eqref{linearized-eqn}, we have the following:
\begin{enumerate}
    \item[a)]Norm equivalence:
    \begin{align*}
        E_{lin}(v)\approx_A\|v\|^2_{L_x^2}
    \end{align*}
    \item[b)]Energy estimates:
    \begin{align*}
        \frac{d}{dt}E_{lin}(v)\lesssim _{A}B^2\|v\|^2_{L_x^2}
    \end{align*}
\end{enumerate}
\end{proposition}

\begin{remark}
It actually turns out that the linearized equation \eqref{linearized-eqn} is well-posed in $L_x^2$. We are not going to use this property in our paper, but we briefly discuss the key ideas behind its proof. The main point is to obtain a similar estimate for the adjoint equation, interpreted as a backward evolution in the space $L^2$. Namely, the adjoint equation corresponding to the linearized one has the form
\begin{align*}
    \partial_tv-2\log|D_x|\partial_xv-Q(\varphi,\partial_xv)=0
\end{align*}
By carrying out a paradifferential normal form transformation, akin to the one from the proof of Proposition \ref{Linearized normalized variable}, we reduce this to the equation
\begin{align*}
    \partial_tv-2\log|D_x|\partial_xv-Q_{lh}(\varphi,\partial_xv)=0.
\end{align*}
By considering the modified energy functional
\begin{align*}
  \int v\cdot T_{\frac{1}{1-\psi_x}}v\,dx, 
\end{align*}
and carrying out an analysis similar to the one from the proof of Proposition \ref{Paradiferential flow linearized energy estimates}, we obtain the desired energy estimate for the dual problem. Now the existence follows by a standard duality argument (for the general theory, see Theorem 23.1.2 in H\" ormander \cite{Hormander}).
\end{remark}
\begin{proof}
Let $E_{lin}(v)=E(\tilde{v})$, where $E(\cdot)$ is defined in Proposition \ref{Paradiferential flow linearized energy estimates}, and $\tilde{v}$ is defined in Proposition \ref{Linearized normalized variable}.

Part a) is immediate, whereas part b) follows from Proposition \ref{Paradiferential flow linearized energy estimates}.
\end{proof}

\section{Higher order energy estimates}\label{s:higher order energy estimates}

In this section we establish higher order energy estimates. Extra care must be taken because commutators with $D^s$ are quadratic rather than cubic, and thus require a normal form correction. 

\begin{proposition}\label{High energy normalized variable}
    Let $s\geq 0$. Given $v$ solving \eqref{paradiff-eqn-inhomog}, there exists a normalized variable $v^s$ such that
    \begin{align*}
        \partial_tv^s - 2\logd \partial_xv^s-\partial_xQ_{lh}(\varphi,v^s)= f + \mathcal{R}(\varphi, v),
    \end{align*}
    with
    \begin{align*}
        \|v^s-|D_x|^sv\|_{L_x^2}&\lesssim_AA\|v\|_{\dot{H}_x^s}
    \end{align*}
    and $\mathcal{R}(\varphi)$ satisfying balanced cubic estimates,
\begin{equation}\label{bal-est-eng}
\|\mathcal{R}(\varphi, v)\|_{L^2} \lesssim_A B^2 \|v\|_{L^2}.
\end{equation}
\end{proposition}

\begin{proof}
Let $v$ satisfy \eqref{paradiff-eqn-inhomog}, where without loss of generality, $f = 0$. The natural approach is to reduce the equation for $v^s := |D_x|^s v$ to \eqref{paradiff-eqn-inhomog} with a perturbative inhomogeneity. However, the commutators arising from such a reduction are quadratic, and cannot satisfy balanced cubic estimates. In particular, they cannot be seen as directly perturbative. We will address these errors via a conjugation combined with a normal form correction.

In preparation, we use Lemmas~\ref{l:toy-red} and \ref{l:q-principal} to rewrite $Q_{lh}$ in \eqref{paradiff-eqn-inhomog} and obtain
\begin{equation}
\begin{aligned}
\D_t v - 2 \D_x (T_{1 - \D_x \psi}\logd + &T_{\logd \D_x \psi + R} + \D_x [T_\psi, \logd]) v = 0.
\end{aligned}
\end{equation}
Then $v^s := |D_x|^s v$ satisfies
\begin{equation}\label{vs-eqn}
\begin{aligned}
\D_t v^s &- 2 \D_x (T_{1 - \D_x \psi}\logd + T_{\logd \D_x \psi + R} + \D_x [T_\psi, \logd]) v^s \\
&= 2 \D_x ([|D_x|^s, T_{- \D_x \psi}]\logd + [|D_x|^s, T_{\logd \D_x \psi + R}] + \D_x[|D_x|^s, [T_\psi, \logd]]) v \\
&=: L(\D_x^2 \psi, \logd v^s) - L(\logd \D_x^2 \psi, v^s) + 2\D_x^2 [|D_x|^s, [T_\psi, \logd]] v + \mathcal R
\end{aligned}
\end{equation}
where we have absorbed $\D_x R$ into $\mathcal R$ and $L$ denotes an order zero paradifferential bilinear form, 
\begin{equation}
\begin{aligned}
L(\D_x f, u) &= -2 \D_x [|D_x|^s, T_f] |D_x|^{-s} u.
\end{aligned}
\end{equation}
In particular, observe that the principal term of $L$ is given by
\[
L(g, u) \approx -2sT_{g} u.
\]

To address the two $L$ contributions, which are quadratic and not directly perturbative, we apply two steps:
\begin{enumerate}
\item[a)] We first apply a conjugation to $v^s$ which improves the leading order of the contributions of the $L$ terms from, up to a logarithm, 0 to $-1$.
\item[b)] We then apply a normal form transformation yielding cubic, balanced source terms.
\end{enumerate}

\

a) We begin by computing the equation for the conjugated variable
\[
\tv^s := T_{J^{-s}} v^s.
\]
To do so, it suffices to apply $T_{J^{-s}}$ to \eqref{vs-eqn} and consider the commutators. These will include a $\D_t$ commutator, a $\D_x$ commutator, and a $\logd$ commutator.

\

i) First, we use \eqref{psi-eqn} to expand the $\D_t$ commutator,
\begin{equation}
\begin{aligned}
\, [T_{J^{-s}}, \D_t] v^s &= - sT_{J^{1-s} \D_x \D_t \psi} v^s = - sT_{J^{1-s}\D_x (2T_{1 - \D_x \psi} \logd \D_x \psi + 2T_{(\logd - 1) \D_x \psi + R} \D_x \psi + f)} v^s.
\end{aligned}
\end{equation}
Here the contribution from $f$ may be estimated using \eqref{psi-eqn-est} and discarded. Further, due to a favorable balance of derivatives when $\D_x$ falls on the lowest frequency variables, we can reduce to 
\[
- 2sT_{J^{-s} (T_J T_{1 - \D_x \psi} \logd \D_x^2 \psi + T_J T_{(\logd - 1) \D_x \psi + R} \D_x^2 \psi)} v^s.
\]
Lastly, we apply Lemma~\ref{l:para-prod} to merge and split paraproducts, reducing to 
\begin{equation}\label{dt-comm}
- 2sT_{\logd \D_x^2 \psi + T_{J(\logd - 1) \D_x \psi + R} \D_x^2 \psi} \tv^s.
\end{equation}
Observe that the first part of \eqref{dt-comm} cancels with the principal term of the second $L$ on the right hand side of \eqref{vs-eqn}. The second part of \eqref{dt-comm} will cancel with a contribution from ii) below.

\

ii) Next, we consider the commutator of $T_{J^{-s}}$ with the outer $\D_x$. We obtain 
\[
2sT_{J^{1-s} \D_x^2 \psi} (T_{1 - \D_x \psi}\logd + T_{\logd \D_x \psi + R} + \D_x [T_\psi, \logd]) v^s.
\]
Here it is convenient to apply \eqref{qlh-expression-full} of Lemma~\ref{l:q-principal} to write this as
\[
2sT_{J^{1-s} \D_x^2 \psi} ((T_{1 - \D_x \psi} \logd + T_{(\logd - 1) \D_x \psi + R}) v^s + \Gamma(\D_x^2 \psi, \D_x^{-1} v^s)).
\]
Applying Lemmas \ref{l:para-prod} and \ref{l:para-com} to split, compose, and commute paraproducts, this may be reduced modulo perturbative terms to 
\[
2sT_{\D_x^2 \psi} \logd \tv^s + 2s T_{ J(\logd - 1) \D_x \psi \cdot \D_x^2 \psi + \D_x^2 \psi \cdot R} \tv^s.
\]
The first term above cancels with the principal term of the first $L$. The second term cancels with the remaining part of the $\D_t$ commutator above in \eqref{dt-comm}. To see this cancellation, we have freely exchanged multiplication by $J(\logd - 1) \D_x \psi$ with a paraproduct, as the difference has a favorable balance of derivatives and is thus perturbative.

\

iii) Returning to the commutator of $T_{J^{-s}}$ with the dispersive term, it remains to consider the commutator with the inner $\log |D_x|$, where we have used Lemma~\ref{l:para-com} to discard any paraproduct commutators. We have
\[
-2\D_x T_{1 - \D_x \psi} [T_{J^{-s}}, \log |D_x|] v^s
\]
whose principal term $2s T_{\D_x^2 \psi} \tv^s$ cancels with the principal term of the double commutator on the right hand side of \eqref{vs-eqn}.

\

To conclude, we have
\begin{equation}\label{tvs-eqn}
\begin{aligned}
\D_t \tv^s &- 2 \D_x (T_{1 - \D_x \psi}\logd + T_{\logd \D_x \psi + R} + \D_x[T_\psi, \logd]) \tv^s \\
&= (L(\D_x^2 \psi, \logd \tv^s) + 2sT_{\D_x^2 \psi} \logd \tv^s) \\
&\quad - (L(\logd \D_x^2 \psi, \tv^s) + 2sT_{\logd \D_x^2 \psi} \tv^s) \\
&\quad + 2T_{J^{-s}}\D_x^2 [|D_x|^s, [T_\psi, \logd]] v - 2\D_x T_{1 - \D_x \psi} [T_{J^{-s}}, \log |D_x|] v^s + f \\
&=: L_0(\D_x^3 \psi, \logd \D_x^{-1}\tv^s) - L_0(\logd \D_x^3 \psi, \D_x^{-1}\tv^s) + L_1(\D_x^3 \psi, \D_x^{-1} \tv^s) + f
\end{aligned}
\end{equation}
where $f$ satisfies \eqref{bal-est-eng}. Here $L_0$ and $L_1$ denote order zero paradifferential bilinear forms, respectively
\begin{equation*}
\begin{aligned}
L_0(\D_x^2 f, \D_x^{-1} u) &= L(\D_x f, u) + 2sT_{\D_x f} u, \\
L_1(\D_x^2 f, \D_x^{-1} T_{J^{-s}} u) &= ( 2T_{J^{-s}}\D_x^2 [|D_x|^s, [T_{\D_x^{-1} f}, \logd]] |D_x|^{-s} u - 2\D_x T_{1 - \D_x \psi} [T_{J^{-s}}, \log |D_x|] u).
\end{aligned}
\end{equation*}
Observe that since $L_i$ are all order 0 paradifferential bilinear forms, we have reduced the terms of the inhomogeneity to order $-1$.

\

b) We next choose a normal form transformation to reduce the quadratic components of the inhomogeneity to balanced cubic terms. Let
\[
\tw^s = \half T_{J} L_0(\D_x^2\psi, \D_x^{-1} \tv^s).
\]
Then we claim that $\tilde u^s := \tv^s - \tilde w^s$ is the desired normal form transform. To see this, it remains to compute 
\begin{equation}\label{compute-nf}
\begin{aligned}
(\D_t - 2 \D_x (T_{1 - \D_x \psi}\logd + T_{\logd \D_x \psi + R} + \D_x[T_\psi, \logd])) \tw^s
\end{aligned}
\end{equation}
and observe cancellation with the three $L_i$ bilinear forms on the right hand side of \eqref{tvs-eqn}. To see this, we partition the computation into the following subgroups:

\

i) When the full equation of \eqref{compute-nf} falls on the high frequency $\tv^s$ input of $L_0$, we may use \eqref{tvs-eqn} to see that the contribution has a favorable balance of derivatives and may be absorbed into $f$. 

\

ii) We may commute the equation freely with the low frequency $J$ due to a favorable balance of derivatives, absorbing the contribution again into $f$. 

\

It remains to consider commutators of the terms of the equation \eqref{compute-nf} across the low frequency $\D_x^2 \psi$ input of $\tilde w^s$. 

\

iii) We first consider the commutators involving the operators
\[
2\D_x(T_{\logd \D_x \psi + R} + \D_x[T_\psi, \logd]).
\] 
We may freely commute the $\D_x$ forward, and also use \eqref{qlh-expression-full}  of Lemma~\ref{l:q-principal} to rewrite, reducing to the operators
\[
2 T_{(\logd - 1) \D_x \psi + R} \D_x + \Gamma(\D_x^2 \psi, \D_x^{-1} (\cdot)) \circ \D_x.
\] 
The contribution from the $\Gamma$ term may be absorbed into $f$ due to a favorable balance. The remaining contribution
\[
T_{J} L_0(T_{(\logd - 1) \D_x \psi + R}\D_x^3\psi, \D_x^{-1} \tv^s)
\]
will cancel with a contribution of step iv) below.

\

iv) For the case when $\D_t$ falls on the low frequency input of $L_0$, we apply equation \eqref{psi-eqn}. Precisely, the two non-perturbative contributions on the left hand side of \eqref{psi-eqn} cancel with the second $L_0$ source term in \eqref{tvs-eqn}, and the remaining contribution of step iii) above, respectively.

\

v) From the dispersive term $2T_{J^{-1}}\D_x \logd$, the case when $\D_x$ falls on the low frequency input of $L_0$ while $\logd$ has commuted to the high frequency input cancels with the first $L_0$ source term in \eqref{tvs-eqn}.

\

vi) From the dispersive term $2T_{J^{-1}}\D_x \logd$, it remains to consider the commutators with $\logd$, where the $\D_x$ remains in front. Using Lemma~\ref{l:para-prod} with $J(1 - \D_x \psi) = 1$, and opening the definition of $L_0$, we have
\[
2\D_x^2 [\logd, [|D_x|^s, T_\psi]] |D_x|^{-s} \tv^s - 2s\D_x [\logd, T_{\D_x \psi}] \tv^s.
\]
We claim that these two terms cancel with the two terms of $L_1$ respectively. Indeed, for the first term, we commute using Lemma~\ref{l:para-com} to reduce to
\[
2T_{J^{-s}}\D_x^2 [\logd, [|D_x|^s, T_\psi]] v
\]
which cancels with the double commutator term of $L_1$. For the second term, we also commute using Lemma~\ref{l:para-com} to reduce to
\[
2s T_{J^{-s}} \D_x [T_{\D_x \psi}, \logd] v^s =: - 2s T_{J^{-s}} T_{\D_x^2 \psi} v^s + T_{J^{-s}} L_2(\D_x^3 \psi, \D_x^{-1} v).
\]
On the other hand, the second term of $L_1$ may be expressed as
\[
- 2\D_x T_{1 - \D_x \psi} [T_{J^{-s}}, \log |D_x|] v^s = 2s T_{1 - \D_x \psi}T_{J^{1-s} \D_x^2 \psi} v^s + T_{1 - \D_x \psi} L_2(J^{1-s} \D_x^3 \psi, \D_x^{-1} v).
\]
These cancel, up to applying Lemma~\ref{l:para-prod} with $J(1 - \D_x \psi) = 1$ and perturbative errors.
\end{proof}

\

We thus obtain the following energy estimate: 

\begin{proposition}\label{Higher energy estimates}
Assume that $A\ll 1$ and $B\in L_t^2$. For every $s\geq 0$, there exist energy functionals $E^{(s)}(v)$ such that we have the following:
\begin{enumerate}
    \item[a)]Norm equivalence:
    \begin{align*}
        E^{(s)}(v)\approx_A\|v\|^2_{\dot{H}_x^s}
    \end{align*}
    \item[b)]Energy estimates:
    \begin{align*}
        \frac{d}{dt}E^{(s)}(v)\lesssim _{A}B^2\|v\|^2_{\dot{H}_x^s}
    \end{align*}
\end{enumerate}
\end{proposition}
\begin{proof}
Let $E^{(s)}(v)=E(v^s)$, where $E(\cdot)$ is defined in Proposition \ref{Paradiferential flow linearized energy estimates}, and $v^s$ is defined in Proposition \ref{High energy normalized variable}.

Part a) is immediate, whereas part b) follows from Proposition \ref{Paradiferential flow linearized energy estimates}.
\end{proof}

\section{Local well-posedness}\label{s:lwp}

To establish the local well-posedness result at low regularity, we follow the approach outlined in \cite{ITprimer}. We consider $\varphi_0 \in \dot{H}_x^{s_1}\cap\dot{H}_x^{s_2}$, with $s_1<\frac{3}{2}$, $s_2>2$. Let $\varphi_0^h=(\varphi_0)_{\leq h}$, where $h\in\mathbb{N}$. Since $\varphi_0^h \rightarrow \varphi_0$ in $\dot{H}_x^{s_1}\cap\dot{H}_x^{s_2}$, we may assume that $ \|\varphi_0^h\|_{\dot{H}_x^{s_1}\cap\dot{H}_x^{s_2}}<R$ for all $h$.

We construct a uniform $\dot{H}_x^{s_1}\cap\dot{H}_x^{s_2} $ frequency envelope $\{c_k\}_{k\in\mathbb{Z}}$ for $\varphi_0$ having the following properties:

\begin{enumerate}
     \item[a)]Uniform bounds:     
     \[ \|P_k(\varphi_0^h)\|_{\dot {H}_x^{s_1}\cap\dot {H}_x^{s_2}}\lesssim c_k,\]
     
     \item[b)]High frequency bounds:     
     \[\|\varphi_0^h\|_{\dot{H}_x^{s_1}\cap\dot{H}_x^N}\lesssim 2^{h(N-s_2)}c_h, \qquad N>s_2, \]
     
     \item[c)]Difference bounds:     
     \[\|\varphi_0^{h+1}-\varphi_0^h\|_{L_x^2}\lesssim 2^{-s_2h}c_h,\]
     
     \item[d)]Limit as $h\rightarrow\infty$:     
     \[ \varphi_0^h\rightarrow \varphi_0 \in \dot{H}_x^{s_1}\cap\dot{H}_x^{s_2}.\]
     
 \end{enumerate}

Let $\varphi^h$ be the solutions with initial data $\varphi_0^h$, whose existence is guaranteed instance by \cite{SQGzero}. Using the energy estimate for the solution $\varphi$ of \eqref{SQG} from Proposition \ref{Higher energy estimates} and Proposition \ref{Normalized variable}, we deduce that there exists $T = T(\|\varphi_0\|_{H_x^s}) > 0$ on which all of these solutions are defined, with high frequency bounds
    \[ 
    \|\varphi^h\|_{C_t^0(\dot{H}_x^{s_1}\cap\dot{H}_x^{N})}\lesssim \|\varphi_0^h\|_{\dot{H}_x^{s_1}\cap\dot{H}_x^{N}} \lesssim 2^{h(N-s_2)}c_h.
     \]
Further, by using the energy estimates for the solution of the linearized equation from Proposition \ref{Linearized equation energy estimates}, we have
\[
\|\varphi^{h+1}-\varphi^h\|_{C_t^0L_x^2}\lesssim 2^{-s_2h}c_h.
\]
By interpolation, we infer that
\[
\|\varphi^{h+1}-\varphi^h\|_{C_t^0(\dot{H}_x^{s_1}\cap\dot{H}_x^{s_2})}\lesssim c_h.
\]

As in \cite{ITprimer}, we get
\[
\|P_k \varphi^h \|_{C_t^0(\dot{H}_x^{s_1}\cap\dot{H}_x^{s_2}) } \lesssim c_k
\]
and that
\[
\|\varphi^{h+k}-\varphi^h\|_{C_t^0(\dot{H}_x^{s_1}\cap\dot{H}_x^{s_2}) }\lesssim c_{h\leq\cdot<h+k}=\left(\sum_{\substack{n=h}}^{h+k-1}c_n^2\right)^{\frac{1}{2}}
\]
for every $k\geq 1$. Thus, $\varphi^h$ converges to an element $\varphi$ belonging to $C_t^0(\dot{H}_x^{s_1}\cap\dot{H}_x^{s_2})([0,T]\times\mathbb{R})$.  Moreover, we also obtain
\begin{equation}\label{convergence estimate}
\begin{aligned}
\|\varphi^h - \varphi\|_{C_t^0(\dot{H}_x^{s_1}\cap\dot{H}_x^{s_2})} &\lesssim c_{\geq h}=\left(\sum_{\substack{n=h}}^{\infty} c_n^2\right)^{\frac{1}{2}}.
\end{aligned}
\end{equation}

We now prove continuity with respect to the initial data. We consider a sequence
 \[
 \varphi_{0j}\rightarrow \varphi_0 \in (\dot{H}_x^{s_1}\cap\dot{H}_x^{s_2})
 \]
 and an associated sequence of $ H_x^s$-frequency envelopes $\{c^j_k\}_{k \in \Z}$, each satisfying the analogous properties enumerated above for $c_k$, and further such that $c^j_k \rightarrow c_k$ in $l^2(\mathbb{Z})$. In particular,
\begin{equation}\label{convergence estimate j}
\|\varphi_j^h - \varphi_j\|_{C_t^0(\dot{H}_x^{s_1}\cap\dot{H}_x^{s_2})} \lesssim c^j_{\geq h}=\left(\sum_{\substack{n=h}}^{\infty} (c^j_n)^2\right)^{\frac{1}{2}}.
\end{equation}
 
Using the triangle inequality with \eqref{convergence estimate} and \eqref{convergence estimate j}, we write
\begin{align*}
\|\varphi_j - \varphi\|_{C_t^0(\dot{H}_x^{s_1}\cap\dot{H}_x^{s_2}) } &\lesssim \|\varphi^h - \varphi\|_{C_t^0(\dot{H}_x^{s_1}\cap\dot{H}_x^{s_2})}+\|\varphi_j^h - \varphi_j\|_{C_t^0(\dot{H}_x^{s_1}\cap\dot{H}_x^{s_2})}+\|\varphi_j^h - \varphi^h\|_{C_t^0(\dot{H}_x^{s_1}\cap\dot{H}_x^{s_2})}\\
&\lesssim c_{\geq h}+c^j_{\geq h}+\|\varphi_j^h - \varphi^h\|_{C_t^0(\dot{H}_x^{s_1}\cap\dot{H}_x^{s_2})}.
\end{align*}
To address the third term, we observe that for every fixed $h$, $\varphi_j^h \rightarrow \varphi^h$ in $(\dot{H}_x^{s_1}\cap\dot{H}_x^{s_2})$. We conclude that $\varphi_j \rightarrow \varphi$ in $C_t^0(\dot{H}_x^{s_1}\cap\dot{H}_x^{s_2}) ([0,T]\times\mathbb{R})$.

\section{Global well-posedness}\label{s:gwp}

In this section we prove global well-posedness for the SQG equation \eqref{SQG} with small and localized initial data. We use the wave packet method of Ifrim-Tataru, which is systematically described in \cite{ITpax}. This section is companion to Section 6 in \cite{SQGzero}.

\subsection{Notation}

Consider the linear flow
\[
i\D_t \varphi - A(D)\varphi = 0
\]
and the linear operator
\[
L = x - tA'(D).
\]
In our setting, we have the symbol
\[
a(\xi) = -2\xi \log |\xi|
\]
and thus
\[
A(D)=-2D\log|D|, \qquad  L = x+2t+2t\log|D|.
\]
We define the weighted energy space ($s_0<1$ and $s>3$)
\[
\|\varphi\|_X = \|\varphi\|_{\dot{H}^{s_0}\cap\dot{H}^s} + \|L\partial_x \varphi\|_{L^2},
\]

We also define the pointwise control norm
\[
\|\varphi\|_Y = \||D_x|^{1-\delta}\varphi\|_{L_x^\infty} + \||D_x|^{\frac12+\delta}\varphi_x\|_{L_x^\infty}.
\]

We partition the frequency space into dyadic intervals $I_\lambda$ localized at dyadic frequencies $\lambda \in 2^\Z$, and consider the associated partition of velocities
\[
J_\lambda = a'(I_\lambda)
\]
which form a covering of the real line, and have equal lengths. To these intervals $J_\lambda$ we select reference points $v_\lambda \in J_\lambda$, and consider an associated spatial partition of unity
\[
1 = \sum_\lambda \chi_\lambda(x), \qquad \text{supp } \chi_\lambda \subseteq  \overline{J_\lambda},\qquad \chi_\lambda=1\text{ on }J_\lambda,
\]
where $\overline{J_\lambda}$ is a slight enlargement of $J_\lambda$, of comparable length, uniformly in $\lambda$.

Lastly, we consider the related spatial intervals, $tJ_\lambda$, with reference points $x_\lambda  = tv_\lambda \in tJ_\lambda$. 
\

\subsection{Overview of the proof}

We provide a brief overview of the proof.

\

1. We make the bootstrap assumption for the pointwise bound
\begin{equation}\label{bootstrap}
    \|\varphi(t)\|_Y \lesssim C \eps \langle t\rangle^{-\frac{1}{2}}
\end{equation}
where $C$ is a large constant, in a time interval $t \in [0, T]$ where $T>1$. 

\

2. The energy estimates for \eqref{SQG} and the linearized equation will imply
\begin{equation}\label{Regular Energy Estimate}
    \|\varphi(t)\|_{X} \lesssim \langle t\rangle^{C^2\eps^2} \|\varphi(0)\|_X.
    \end{equation}
    
\

3. We aim to improve the bootstrap estimate \eqref{bootstrap} to 
\begin{equation}\label{pointwise}
    \|\varphi(t)\|_Y \lesssim \eps \langle t\rangle^{-\frac12}.
\end{equation}
We use vector field inequalities to derive bounds of the form
\begin{equation}\label{pt-estimate}
    \|\varphi(t)\|_Y\lesssim \eps \langle t\rangle^{-\frac{1}{2} + C\eps^2},
\end{equation}
which is the desired bound but with an extra $t^{C\eps^2}$ loss.

\

4. In order to rectify the extra loss, we use the wave packet testing method. Namely, we define a suitable asymptotic profile $\gamma$, which is then shown to be an approximate solution for an ordinary differential equation. This enables us to obtain suitable bounds for the asymptotic profile without the aforementioned loss, which can then be transferred back to the solution $\varphi$.

\subsection{Energy estimates} 

From Proposition \ref{Higher energy estimates} and Gr\"onwall's lemma, together with the fact that $\eps\ll 1$, we get that
\begin{align*}
 \|\varphi(t,x)\|_{H_x^s}\lesssim e^{C\int_0^tC(A(\tau))B(\tau)^2\,d\tau}\|\varphi_0\|_{H_x^s}. 
\end{align*}

Let $u=L\partial_x\varphi+t\int F(\dq^y\varphi)\sdq^y\varphi_x\,dy$, which satisfies the linearized equation with error $\int F'(\dq^y\varphi)\dq^y\varphi\sdq^y\varphi_x\,dy$, which is clearly balanced. From Proposition \ref{Linearized equation energy estimates}, along with Gr\"onwall's lemma and the fact that $\eps\ll 1$, we have
\begin{align*}
 \|u(t,x)\|_{L_x^2}\lesssim e^{C\int_0^tC(A(\tau))B(\tau)^2\,d\tau}\|u_0\|_{L_x^2}. 
\end{align*}
Along with the bootstrap assumptions, these readily imply that
\begin{equation}\label{Vector Field Energy Estimate}
    \|\varphi\|_{X}\lesssim \|\varphi(t)\|_{H_x^s} + \|u(t)\|_{L_x^2}\lesssim \eps e^{C^2\eps^2\int_0^t\langle s\rangle^{-1}\,ds}\lesssim \eps \langle t\rangle^{C^2\eps^2}.
\end{equation}

\

\subsection{Vector field bounds} 
Proposition 2.1 from \cite{ITpax} implies that
\begin{align*}
    \|\varphi_\lambda\|^2_{L_x^{\infty}}&\lesssim \frac{1}{t}(\|\varphi_\lambda\|_{L_x^2}\|L\partial_x\varphi_\lambda\|_{L_x^2}+\|\varphi_\lambda\|^2_{L_x^2}).
\end{align*}
When $\lambda\leq 1$,
\begin{align*}
  \|\varphi_\lambda\|_{L_x^{\infty}}
  &\lesssim \frac{1}{\sqrt{t}}\lambda^{-(1-\delta-\delta_1)}(\|\lambda^{2-2\delta-2\delta_1}\varphi_\lambda\|^{1/2}_{L_x^2}\|L\partial_x\varphi_\lambda\|^{1/2}_{L_x^2}+\|\lambda^{1-\delta-\delta_1}\varphi_\lambda\|_{L_x^2}) \lesssim \frac{1}{\sqrt{t}}\lambda^{-(1-\delta-\delta_1)}\|\varphi\|_X
\end{align*}
and when $\lambda>1$,
\begin{align*}
  \|\varphi_\lambda\|_{L_x^{\infty}}
  &\lesssim \frac{1}{\sqrt{t}}\lambda^{-\frac32 - 2 \delta}(\|\lambda^{3+4\delta}\varphi_\lambda\|^{1/2}_{L_x^2}\|L\partial_x\varphi_\lambda\|^{1/2}_{L_x^2}+\|\lambda^{\frac32+2\delta}\varphi_\lambda\|_{L_x^2}) \lesssim \frac{1}{\sqrt{t}}\lambda^{-\left(\frac32+2\delta\right)}\|\varphi\|_X.
\end{align*}
By dyadic summation and Bernstein's inequality, we deduce the bound
\begin{equation}\label{Pointwise Vector Field Bound 2}
    \|\varphi\|_Y=\|\langle D_x\rangle^{\frac12+2\delta}|D_x|^{1-\delta}\varphi\|_{L_x^{\infty}}\lesssim\frac{\|\varphi\|_X}{\sqrt{t}}.
\end{equation}

By the localized dispersive estimate \cite[Proposition 5.1]{ITpax}, 
\begin{align*}
    |\varphi_\lambda(x)|^2\lesssim\frac{1}{|x-x_{\lambda}|t\frac{1}{\lambda}}(\|L\varphi_\lambda\|_{L_x^2}+\lambda^{-1}\|\varphi_\lambda\|_{L_x^2})^2,
\end{align*}
 which implies that
 \begin{equation}\label{Pointwise Elliptic Estimate}
    \|(1-\chi_{\lambda})\varphi_\lambda\|_{L_x^{\infty}}\lesssim \frac{\lambda^{-1/2}}{t}(\|L\partial_x\varphi_\lambda\|_{L_x^2}+\|\varphi_\lambda\|_{L_x^2})\lesssim\frac{\lambda^{-1/2}}{t}(\|\varphi\|_X+\|\varphi_\lambda\|_{L_x^2})
\end{equation}
 
\

To end this section we record the following elliptic bounds:

\begin{lemma}\label{Elliptic bounds for the derivative}
We have
\begin{equation}
\label{Pointwise Elliptic Estimate 3}
    \||D_x|^{\frac12+\delta}\partial_x((1-\chi_{\lambda})\varphi_\lambda)\|_{L_x^{\infty}}\lesssim \frac{\lambda^{1+\delta}}{t}(\|\varphi\|_X+\|\varphi_\lambda\|_{L_x^2})
    \end{equation}
    \begin{equation}
\label{Pointwise Elliptic Estimate 2}
   \||D_x|^{1-\delta}((1-\chi_{\lambda})\varphi_\lambda)\|_{L_x^{\infty}}\lesssim \frac{\lambda^{\frac12-\delta}}{t}(\|\varphi\|_X+\|\varphi_\lambda\|_{L_x^2}),
\end{equation}
and
\begin{equation}\label{L2 Elliptic Estimate}
\|(1-\chi_{\lambda})\varphi_\lambda\|_{L_x^{2}}\lesssim \frac{\lambda^{-1}}{t}(\|\varphi\|_X+\|\varphi_\lambda\|_{L_x^2}),
\end{equation}
Moreover, the difference quotient satisfies the bounds
  \begin{align*}
\|(1-\chi_{\lambda})\dq^y\varphi_\lambda\|_{L_x^{\infty}}&\lesssim  \frac{\lambda^{1/2}}{t}(\|\varphi\|_X+\|\varphi_\lambda\|_{L_x^2}),
  \end{align*}
  and
  \begin{align*}
\|(1-\chi_{\lambda})\dq^y\varphi_\lambda\|_{L_x^{2}}&\lesssim  \frac{(\|\varphi\|_X+\|\varphi_\lambda\|_{L_x^2})}{t}.
  \end{align*}
\end{lemma}
\begin{proof}
We use the bounds
\begin{align*}
    |\partial_x(\chi_{\lambda}(x/t)|&\lesssim t^{-1}.
\end{align*}
From \ref{Pointwise Elliptic Estimate} applied for $\partial_x\varphi$,
\begin{align*}
    \|\partial_x((1-\chi_{\lambda})\varphi_\lambda)\|_{L_x^{\infty}}&\lesssim \frac{1}{t}\|\chi'_{\lambda}\varphi_\lambda\|_{L_x^{\infty}}+\|(1-\chi_{\lambda})\partial_x\varphi_\lambda\|_{L_x^{\infty}} \lesssim \frac{\lambda^{1/2}}{t}(\|\varphi\|_X+\|\varphi_\lambda\|_{L_x^2}).
\end{align*}
The first two bounds immediately follow from \ref{Pointwise Elliptic Estimate}, and the $L^2$ elliptic estimate similarly follows from \cite[Proposition 5.1]{ITpax}.

For the bounds involving the difference quotient, from \ref{Pointwise Elliptic Estimate} applied for $\dq^y\varphi$, we have
\begin{align*}
\|(1-\chi_{\lambda})\dq^y\varphi_\lambda\|_{L_x^\infty}&\lesssim \frac{\lambda^{1/2}}{t}(\|L\dq^y\varphi_\lambda\|_{L_x^2}+\lambda^{-1}\|\dq^y\varphi_\lambda\|_{L_x^2})\\
&\lesssim \frac{\lambda^{1/2}}{t}(\|\dq^y(L\varphi_\lambda)\|_{L_x^2}+\|\varphi_\lambda(x+y)\|_{L_x^2}+\|\varphi_\lambda\|_{L_x^2})\\
&\lesssim \frac{\lambda^{1/2}}{t}(\|L\partial_x\varphi_\lambda\|_{L_x^2}+\|\varphi_\lambda\|_{L_x^2})\\
&\lesssim \frac{\lambda^{1/2}}{t}(\|\varphi\|_X+\|\varphi_\lambda\|_{L_x^2})
\end{align*}
The other bound is proved similarly.
\end{proof}

\subsection{Wave packets}

We construct wave packets as follows. Given the dispersion relation $a(\xi)$, the group velocity $v$ satisfies
\[
v = a'(\xi) = -2-2\log|\xi|,
\]
so we denote
\[
\xi_v = -e^{-1-\frac{v}{2}}.
\]
Then we define the linear wave packet $\pax^v$ associated with velocity $v$ by
\[
\pax^v = a''(\xi_v)^{-\half} \chi(y) e^{it\phi(x/t)}, \qquad y = \frac{x - vt}{t^\half a''(\xi_v)^\half},
\]
where the phase $\phi$ is given by
\[
\phi(v) = v\xi_v - a(\xi_v),
\]
and $\chi$ is a unit bump function, such that $\int\chi(y)\,dy=1$.

We remark that we will typically apply frequency localizations of the form $\pax^v_\lambda = P_\lambda\mathbf{u}^v$ with $v \in J_\lambda$.

\

We observe that since
\[
\D_v (|\xi_v|^{\half} ) = -\frac14 |\xi_v|^{\half}, \qquad \D_v (a''(\xi_v)^{-\half} ) = -\frac14 a''(\xi_v)^{-\half},
\]
we may write
\begin{equation}\label{dpax}
\D_v\pax^v = - \tilde L \pax^v + \pax^{v,II} = t^{\half} a''(\xi_v)^{-\half} \pax^v + \pax^{v,II} 
\end{equation}
where
 \[ \tilde{L}=t(\partial_x-i\phi'(x/t))\]
and $\pax^{v,II}$ has a similar wave packet form. We also recall from \cite[Lemmas 4.4, 5.10]{ITpax} the sense in which $\pax^v$ is a good approximate solution:

\begin{lemma}\label{l:paxsoln}
The wave packet $\pax^v$ solves an equation of the form
\[
(i\D_t - A(D)) \pax^v = t^{-\frac32}(L \pax^{v, I} + \rpax^v)
\]
where $\pax^{v, I}, \rpax^v$ have wave packet form,
\[
\pax^{v, I} \approx a''(\xi_v)^{-\half} \pax^v, \qquad \rpax^v \approx \xi_v^{-1}a''(\xi_v)^{-\half} \pax^v.
\]
\end{lemma}

\

The asymptotic profile at frequency $\lambda$ is meaningful when the associated spatial region $tJ_\lambda$ dominates the wave packet scale at frequency $\lambda$:
\[
\delta x \approx t^\half a''(\lambda)^\half \lesssim |tJ_\lambda| \approx t \lambda a''(\lambda).
\]
This corresponds to 
\[
t \gtrsim \lambda^{-2} a''(\lambda)^{-1} \approx \lambda^{-1}.
\]
Accordingly we define
\[
\mathcal D = \{(t, v) \in \R^+ \times \R : v \in J_\lambda, \ t \gtrsim \lambda^{-1} \}.
\]

\subsection{Wave packet testing}

In this section we establish estimates on the asymptotic profile function
\[
\gamma^\lambda(t, v) := \langle \varphi, \mathbf{u}^v_\lambda \rangle_{L^2_x}=\langle \varphi_\lambda, \mathbf{u}^v \rangle_{L^2_x}.
\]
We will see that $\gamma^\lambda$ essentially has support $v\in J_\lambda$.

\

We will also use the following crude bounds involving the higher regularity of $\gamma^\lambda$:
\begin{lemma}[Lemma 6.3, \cite{SQGzero}]\label{Gamma bounds}
We have
    \begin{align*}
    \|\chi_\lambda \D_v^n \gamma^\lambda\|_{L^\infty}&\lesssim  t^{\half}(1+ t^\half\lambda^\half)^n\|\varphi_{\lambda}\|_{L_x^{\infty}}, \\
      \|\chi_\lambda \D_v^n\gamma^\lambda\|_{L^2}&\lesssim  (t\lambda)^{\frac14}(1+t^\half\lambda^\half)^n\|\varphi_{\lambda}\|_{L_x^{2}},
    \end{align*}
    and 
    \begin{align*}
\|\chi_\lambda\partial_v \gamma^\lambda\|_{L^\infty}&\lesssim t^{\frac14}\lambda^{-\frac34}\|\varphi\|_X + t^\half \|\varphi_\lambda\|_{L_x^{\infty}}.
    \end{align*}
    
\end{lemma}

The first two bounds are reflective of the fact that the pointwise and energy bounds can be transferred from $\varphi_\lambda$ to the approximate profile corresponding to the dyadic frequency $\lambda$. The last bound is more special, as it makes use of the fact that $\gamma^\lambda$ is defined in terms of the wave packet $\mathbf{u}^v$, in order to provide a pointwise estimate that uses the localized energy norm.

\subsubsection{Approximate profile}

We recall from \cite{ITpax} that $\gamma^\lambda$ provides a good approximation for the profile of $\varphi$. In our setting, we will also need to compare the profile with the differentiated flow $\D_x \varphi$. Define
\[
r^\lambda(t, x) = \chi_\lambda(x/t)\varphi_\lambda(t, x) - t^{-\half}\chi_\lambda(x/t)\gamma^\lambda(t, x/t)e^{-it\phi(x/t)}.
\]

\begin{lemma}[Lemma 6.4, \cite{SQGzero}]\label{Error bounds 1}
Let $t \geq 1$. Then we have
\begin{equation*}
\begin{aligned}
\|\chi_\lambda(x/t) r^\lambda\|_{L^\infty_x} &\lesssim t^{-\frac34}\lambda^{-\frac14} \| \tilde L \varphi_\lambda\|_{L^2_x}, \\
\|\chi_\lambda(x/t) \D_v r^\lambda\|_{L^\infty_x} &\lesssim t^{\frac14}\lambda^{-\frac14} \| \tilde L \D_x\varphi_\lambda\|_{L^2_x} +  (1 + t^\half \lambda^\half)\|\varphi_\lambda \|_{L^\infty}.
\end{aligned}
\end{equation*}
\end{lemma}

We also observe that on the wave packet scale, we may replace $\gamma(t, v)$ with $\gamma(t, x/t)$ up to acceptable errors. Denote
\[
\beta^\lambda_v(t, x) = t^{-1/2}\chi_\lambda(x/t)(\gamma(t, v) - \gamma(t, x/t))e^{it\phi(x/t)},
\]

\begin{lemma}[Lemma 6.5, \cite{SQGzero}]\label{Beta bounds} 
Let $v\in J_\lambda$, and $(t,v)\in\mathcal{D}$. Then, for every $y\neq 0$ and $x$ such that $\displaystyle|x-vt|\lesssim\delta x=t^{1/2}\lambda^{-1/2}$, we have the bound 
\begin{align*}
|\dq^y\beta_v|&\lesssim t^{-3/4}\lambda^{-1/4}\|\varphi\|_X
\end{align*}
\end{lemma}
Lemmas \ref{Error bounds 1} and \ref{Beta bounds} are proved by using the wave packet definition of the asymptotic profile $\gamma^\lambda$, as well as the wave packet representation of $\partial_v\mathbf{u}_v$ (\eqref{dpax}), in order to estimate the two errors in regular (Sobolev) or localized energy spaces.

\subsection{Bounds for $Q$}

Write, slightly abusing notation, 
\[
Q(\varphi) = Q(\varphi, \overline{\varphi}, \varphi) := \frac13 \int \sgn(y) \cdot |\dq^y\varphi|^2\dq^y\varphi\, dy.
\]

We recall from \cite{SQGzero} the following lemma:
\begin{lemma}[See Lemma 6.6,\cite{SQGzero}]\label{Cubic difference bound}
For $\displaystyle 0<\delta\ll 1$, we have the difference estimates 
\begin{equation*}\begin{aligned}
\|Q(\varphi_1) - Q(\varphi_2)\|_{L_x^\infty + L_x^{1/\delta}} &\lesssim (\|\D_x (\varphi_1, \varphi_2)\|_{L_x^\infty}^2 +\|\D_x(\varphi_1,\varphi_2)\|
_{L_x^{\frac{1}{2\delta}}}\|(\varphi_1,\varphi_2)\|_{L_x^\infty}) \|\partial_x(\varphi_1-\varphi_2)\|_{L_x^{\infty}},
\end{aligned}\end{equation*}
\begin{equation*}\begin{aligned}
\|Q(\varphi_1) - Q(\varphi_2)\|_{L_x^2}&\lesssim \|\D_x (\varphi_1,\varphi_2)\|_{L_x^2}\|\D_x(\varphi_1,\varphi_2)\|_{L_x^\infty}\|\partial_x(\varphi_1-\varphi_2)\|_{L_x^\infty}\\
&+\||D_x|^{1-\delta}(\varphi_1,\varphi_2)\|_{L_x^2}\|(\varphi_1,\varphi_2)\|_{L_x^\infty}\|\partial_x(\varphi_1-\varphi_2)\|_{L_x^\infty} ,
\end{aligned}\end{equation*}
\begin{equation*}\begin{aligned}
\|Q(\varphi_1) - Q(\varphi_2)\|_{L_x^\infty + L_x^{1/\delta}} &\lesssim (\|\langle D_x\rangle^{\delta} (\varphi_1, \varphi_2)\|_{L_x^\infty}\|\partial_x(\varphi_1, \varphi_2)\|_{L_x^\infty+L_x^{\frac{1}{2\delta}}}\|\partial_x(\varphi_1-\varphi_2)\|_{L_x^{\infty}},
\end{aligned}\end{equation*}
\begin{equation*}\begin{aligned}
\|Q(\varphi_1) - Q(\varphi_2)\|_{L_x^2} &\lesssim (\|\langle D_x\rangle^{\delta} (\varphi_1, \varphi_2)\|_{L_x^\infty}\|\partial_x(\varphi_1, \varphi_2)\|_{L_x^2}\|\partial_x(\varphi_1-\varphi_2)\|_{L_x^{\infty}}.
\end{aligned}\end{equation*}
\end{lemma}
We note that only the first two estimates were proved in \cite{SQGzero}. However, we omit the proofs for the other two, as they are similar.

We will be considering separately the balanced and unbalanced components of $Q$. Precisely, we denote the diagonal set of frequencies by $\mathcal D$ and write
\begin{equation*}
\begin{aligned}
Q(\varphi, \varphi, \varphi) &= \sum_{(\lambda_1, \lambda_2, \lambda_3, \lambda) \in \mathcal D} Q(\varphi_{\lambda_1}, \varphi_{\lambda_2}, \varphi_{\lambda_3}) + \sum_{(\lambda_1, \lambda_2, \lambda_3, \lambda) \notin \mathcal D} Q(\varphi_{\lambda_1}, \varphi_{\lambda_2}, \varphi_{\lambda_3}) \\
&= Q^{bal}(\varphi, \varphi, \varphi) + Q^{unbal}(\varphi, \varphi, \varphi) = Q^{bal}(\varphi) + Q^{unbal}(\varphi).
\end{aligned}
\end{equation*}

\

The unbalanced portion of $Q$ satisfies the better bound as follows: 

\begin{lemma}\label{Unbalanced cubic estimate}
$Q^{unbal}$ satisfies the bounds
\[
\|\chi_\lambda^1\partial_xP_\lambda Q^{unbal}(\varphi)\|_{L_x^\infty} \lesssim \lambda^{-\delta}\frac{\|\varphi\|^3_X}{t^2}
\]
and
\[
\|\chi_\lambda^1\partial_xP_\lambda Q^{unbal}(\varphi)\|_{L_x^2} \lesssim \lambda^{-\delta}\frac{\|\varphi\|^3_X}{t^{3/2}},
\]
where $\chi^1_{\lambda}$ is a cut-off widening $\chi_{\lambda}$.
\end{lemma}

\begin{proof}
We shall denote
\begin{align*}
    I_{\lambda_1,\lambda_2,\lambda_3}&=\int_{\mathbb{R}}\sgn(y)\dq^y\varphi_{\lambda_1}\dq^y\varphi_{\lambda_2}\dq^y\varphi_{\lambda_3}\,dy
\end{align*}
and consider two cases in the frequency sum for $\partial_x P_\lambda Q^{unbal}$. 

First we consider the case in which we have two low separated frequencies. We assume without loss of generality that $\lambda_3=\lambda$ and $\lambda_1<\lambda_2\ll \lambda$. In this case, the elliptic estimates will be applied for the factor $\varphi_{\lambda_1}$. Precisely, from Lemma \ref{Trilinear integral estimate} and estimates \ref{Pointwise Vector Field Bound 2}, \ref{Pointwise Elliptic Estimate}, and \ref{L2 Elliptic Estimate}, we get that
\begin{align*}
    \left\|\chi^1_{\lambda}I_{\lambda_1,\lambda_2,\lambda_3}\right\|_{L_x^{\infty}}&\lesssim \lambda_1\frac{\lambda_1^{-1/2}}{t}\|\varphi\|_X(\lambda_2^{1-2\delta}+\lambda_2)\|\varphi_{\lambda_2}\|_{L_x^\infty}\lambda_3^{\delta}\|\varphi_{\lambda_3}\|_{L_x^\infty}\\
    &\lesssim \frac{\lambda_1^{1/2}}{t}\|\varphi\|_X\lambda_2^{\delta}(\lambda_2^{1-3\delta}+\lambda_2^{1-\delta})\|\varphi_{\lambda_2}\|_{L_x^\infty}\lambda^{-\frac32-2\delta}\lambda^{\frac32+3\delta}\|\varphi_{\lambda}\|_{L_x^\infty}\\
    &\lesssim \lambda_1^{1/2}\lambda_2^{\delta}\lambda^{-\frac32-2\delta}\frac{\|\varphi\|^3_X}{t^2}.
\end{align*}
By using dyadic summation in $\lambda_1$ and $\lambda_2$, we deduce that 

\begin{align*}
    \left\|\chi^1_{\lambda}\partial_x\sum_{\substack{
    \lambda_1<\lambda_2\ll \lambda}}I_{\lambda_1,\lambda_2,\lambda_3}\right\|_{L_x^{\infty}}&\lesssim \lambda^{-\delta}\frac{\|\varphi\|_X^3}{t^{2}}.
\end{align*}
Similarly, we deduce that 
\begin{align*}
    \left\|\chi^1_{\lambda}\partial_x\sum_{\substack{
    \lambda_1<\lambda_2\ll \lambda}}I_{\lambda_1,\lambda_2,\lambda_3}\right\|_{L_x^{2}}&\lesssim \lambda^{-\delta}\frac{\|\varphi\|_X^3}{t^{3/2}}
\end{align*}

We now analyze the situation in which $\lambda_1,\lambda_2\gtrsim \lambda$, and $\lambda_1$ and $\lambda_2$ are comparable and both separated from $\lambda$. Thus, we will be able to use $\lambda_1$ and $\lambda_2$ interchangeably.  We replace $\chi^1_{\lambda}$ by $\tilde{\chi}_{\lambda}$, which has double support, and equals $1$ on a comparably-sized neighbourhood of the support of $\chi^1_{\lambda}$. We write
\begin{align*}
    \chi^1_{\lambda}\partial_xP_\lambda=\chi^1_{\lambda}\partial_xP_\lambda\tilde{\chi}_{\lambda}+\chi^1_{\lambda}\partial_xP_\lambda(1-\tilde{\chi}_{\lambda}).
\end{align*}

For the first term, using Lemma \ref{Trilinear integral estimate}, along with estimates \ref{Pointwise Vector Field Bound 2}, \ref{Pointwise Elliptic Estimate}, \ref{L2 Elliptic Estimate}, we get the bounds 
\begin{align*}
    \left\|\chi^1_{\lambda}P_\lambda\tilde{\chi}_{\lambda}I_{\lambda_1,\lambda_2,\lambda_3}\right\|_{L_x^{\infty}}&\lesssim \lambda_2^{1/2+\delta}\frac{\lambda_3^{1-2\delta}+\lambda_3}{t}\|\varphi\|_X\|\varphi_{\lambda_2}\|_{L_x^\infty}\|\varphi_{\lambda_3}\|_{L_x^\infty}\\
    &\lesssim \lambda_2^{-1-\delta/2}\lambda_3^{\delta/2}\frac{\|\varphi\|_X}{t}(\lambda_3^{1-5\delta/2}+\lambda_3^{1-\delta/2})\|\varphi_{\lambda_3}\|_{L_x^\infty}\lambda_2^{3/2+3\delta/2}\|\varphi_{\lambda_2}\|_{L_x^\infty}\\
    &\lesssim\lambda_2^{-1-\delta/2}\lambda_3^{\delta/2}\frac{\|\varphi\|^3_X}{t^2}
\end{align*}
and
\begin{align*}
    \left\|\chi^1_{\lambda}P_\lambda\tilde{\chi}_{\lambda}I_{\lambda_1,\lambda_2,\lambda_3}\right\|_{L_x^{2}}&\lesssim \lambda_2^{\delta}\frac{\lambda_3^{1-2\delta}+\lambda_3}{t}\|\varphi\|_X\|\varphi_{\lambda_2}\|_{L_x^\infty}\|\varphi_{\lambda_3}\|_{L_x^\infty}\\
    &\lesssim \lambda_2^{-1-3\delta/2}\lambda_3^{\delta/2}\frac{\|\varphi\|_X}{t}(\lambda_3^{1-5\delta/2}+\lambda_3^{1-\delta/2})\|\varphi_{\lambda_3}\|_{L_x^\infty}\lambda_2^{1+5\delta/2}\|\varphi_{\lambda_2}\|_{L_x^\infty}\\
    &\lesssim\lambda_2^{-1-3\delta/2}\lambda_3^{\delta/2}\frac{\|\varphi\|^3_X}{t^2}.
\end{align*}
By using dyadic summation in $\lambda_1$, $\lambda_2$, and $\lambda_3$ (and by using the fact that $\lambda_1$ and $\lambda_2$ are close), we deduce the bound
\begin{align*}
   \left\|\chi^1_{\lambda}\partial_xP_\lambda\tilde{\chi}_{\lambda}\sum_{\substack{
    \lambda_3\lesssim \lambda_2,\lambda_1\simeq\lambda_2\gtrsim \lambda}}I_{\lambda_1,\lambda_2,\lambda_3}\right\|_{L_x^\infty\cap L_x^{2}}&\lesssim  \lambda^{-\delta}\frac{1}{t^2}\|\varphi\|^3_X. 
\end{align*}

We look at the second term. For every $N$, we know that
\begin{align*}
    \|\chi^1_{\lambda}\partial_xP_\lambda(1-\tilde{\chi}_{\lambda})\|_{L^2\rightarrow L^2}, \|\chi^1_{\lambda}\partial_xP_\lambda(1-\tilde{\chi}_{\lambda})\|_{L^\infty\rightarrow L^\infty}&\lesssim \frac{\lambda^{1-N}}{t^{N}}
\end{align*}
We take $N=\frac{3}{2}$. By carrying out a similar analysis as above, along with Lemma \ref{Trilinear integral estimate} and dyadic summation, we deduce that the contributions corresponding to these terms are also acceptable. 
\end{proof}
\begin{lemma}[Lemma 6.8 \cite{SQGzero}]\label{Semiclassical computation}
We have
\begin{align*}\chi_\lambda((x/t))^3Q(e^{it\phi(x/t)})&=(\chi_\lambda(x/t))^3e^{it\phi(x/t)}q(\phi'(x/t))+h(\lambda,t),
\end{align*}
where for every $a\in(0,1)$
\begin{align*}
|h(\lambda,t)|&\lesssim \frac{\lambda^3}{t^{2-3a}}+\frac{\lambda^2}{t^{1-a}}+\frac{1}{t^{2a}}
\end{align*}
\end{lemma}

\
This result can be viewed as a semiclassical expansion of the cubic form $Q$ applied to the wave packet phase correction, and will be useful in deriving an asymptotic ordinary differential equation for the profile $\gamma^\lambda$.
\subsection{The asymptotic equation for $\gamma$}

Here we record the error bounds for the asymptotic equation for $\gamma$. The proof follows precisely that of Proposition 6.9 in \cite{SQGzero}. The change in exponents corresponds to the one in Lemma \ref{Unbalanced cubic estimate}.

\begin{proposition}\label{Asymptotic equation}
Let $v\in J_\lambda$. Under the assumption $(t,v)\in\mathcal{D}$, we have
\begin{align*}
    \dot{\gamma}(t,v)=iq(\xi_v)\xi_vt^{-1}\gamma(t,v)|\gamma(t,v)|^2+f(t,v),
\end{align*}
where
\begin{align*}
    |f(t,v)| &\lesssim \lambda^{-\delta}g(\lambda)t^{-1-\delta+C\eps^2}\eps,
\end{align*}
where $g(\lambda)$ is a finite sum of powers of $\lambda$, and
\begin{align*}
    \|f(t,v)\|_{L_v^2(J_{\lambda})}&\lesssim \lambda^{-\delta}(1+\lambda^{-\half})t^{-1-\delta+C\eps^2}\eps.
\end{align*}
\end{proposition}

\subsection{Closing the bootstrap argument}
We recall that
\begin{align*}
  \|\varphi_\lambda\|_{L_x^{\infty}}
  &\lesssim \frac{1}{\sqrt{t}}\lambda^{-(1-\delta-\delta_1)}\|\varphi\|_X\lesssim \frac{1}{\sqrt{t}}\lambda^{-(1-\delta-\delta_1)}\eps t^{C\eps^2},
\end{align*}
when $\lambda\leq 1$
and 
\begin{align*}
  \|\varphi_\lambda\|_{L_x^{\infty}}
  \lesssim \frac{1}{\sqrt{t}}\lambda^{-(\frac32+3\delta/2)}\|\varphi\|_X\lesssim\frac{1}{\sqrt{t}}\lambda^{-(\frac32+3\delta/2)}\epsilon t^{C\epsilon^2},
\end{align*}
when $\lambda>1$.

Thus, if $\displaystyle t\lesssim \lambda^{N}$ when $\lambda>1$, and if $\displaystyle t\lesssim \lambda^{-N}$ when $\lambda\leq 1$, where $N$ can be chosen appropriately, we get the desired bounds. We are left to analyze the cases $\displaystyle t\gtrsim \lambda^{N}$ when $\lambda>1$, and $\displaystyle t\gtrsim \lambda^{-N}$ when $\lambda\leq 1$.

We recall that in the elliptic region,
\begin{align*}
    \||D_x|^{\delta}((1-\chi_{\lambda})\varphi_\lambda(x))\|_{L_x^{\infty}}&\lesssim \frac{\lambda^{\delta-1/2}}{t}(\|\varphi\|_X+\|\varphi_\lambda\|_{L_x^2})\lesssim\frac{\lambda^{1/2-\delta}+\lambda^{-\delta}}{t}\eps t^{C^2\eps^2}\\
    \||D_x|^{\frac12+\delta}\partial_x((1-\chi_{\lambda})\varphi_\lambda(x))\|_{L_x^{\infty}}&\lesssim \frac{\lambda^{1+\delta}}{t}(\|\varphi\|_X+\|\varphi_\lambda\|_{L_x^2})\lesssim\frac{\lambda^{1+\delta}+\lambda^{1/2+\delta}}{t}\eps t^{C^2\eps^2}
\end{align*}
which give the desired bounds in both cases. It remains to bound the non-elliptic region $\displaystyle \chi_{\lambda}\varphi_\lambda$.
We recall that, if $x/t\in J_{\lambda}$, and $\displaystyle r(t,x)=\chi_{\lambda}\varphi_\lambda(t,x)-\frac{1}{\sqrt{t}}\chi_{\lambda}\gamma(t,x/t)e^{it\phi(x/t)}$,
\begin{align*}
   t^{1/2}\|r^\lambda\|_{L_x^{\infty}}&\lesssim t^{-1/4}\lambda^{-5/4}\eps t^{C\eps^2} 
\end{align*}

When $\lambda\leq 1$ and $t\gtrsim \lambda^{-N}$, we note that
\begin{align*}
  t^{-1/4}\lambda^{-5/4}\eps t^{C\eps^2}&\lesssim \lambda^{-(1-\delta-\delta_1)}\eps.   
\end{align*}
 This is true because it is equivalent to
\begin{align*}
    \lambda^{-(1/4+\delta+\delta_1)}\lesssim t^{1/4-C\eps^2}.
\end{align*}

When $\lambda>1$ and $t\gtrsim \lambda^N$, we can see that
\begin{align*}
  t^{-1/4}\lambda^{-5/4}\eps t^{C\eps^2}&\lesssim \lambda^{-(\frac32+3\delta/2)}\eps.   
\end{align*}
This is true because it is equivalent to
\begin{align*}
    \lambda^{1/4+3\delta/2}\lesssim t^{1/4-C\eps^2}.
\end{align*}

This means that we only need the bounds
\begin{align*}
   |\gamma(t,v)|\lesssim\eps \lambda^{-(1-\delta-\delta_1)} 
\end{align*}
when $\lambda\leq 1$, and
\begin{align*}
    |\gamma(t,v)|\lesssim\eps \lambda^{-(\frac32+3\delta/2)} 
\end{align*}
when $\lambda>1$.
By initializing at time $t=1$, up to which the bounds are known to be true from the energy estimates, and by using Proposition \ref{Asymptotic equation}, we reach the desired conclusion.

\section{Modified scattering}\label{s:scattering}
 In this section we discuss the modified scattering behaviour of the global solutions constructed in Section \ref{s:gwp}. We recall that the solutions of \eqref{SQG} have conserved mass (as llng as it is well-defined):
\begin{proposition}[Proposition 7.1, \cite{SQGzero}]\label{Conservation of mass}
   For solutions $\varphi$ of \eqref{SQG}, $\|\varphi(t)\|^2_{L^2}$ is conserved in time.  
\end{proposition}

Recall the asymptotic equation
\begin{align*}
    \dot{\gamma}(t,v)&=iq(\xi_v)\xi_v t^{-1}\left|\gamma(t,v)\right|^2\gamma(t,v)+f(t,v),
\end{align*}

As $t\rightarrow\infty$, $\gamma(t,v)$ converges to the solution of the equation
\begin{align*}
    \dot{\tilde{\gamma}}(t,v)&=iq(\xi_v)\xi_vt^{-1}\tilde{\gamma}(t,v)|\tilde{\gamma}(t,v)|^2,
\end{align*}
whose solution is 
\begin{align*}
    \tilde{\gamma}(t,v)&=W(v)e^{iq(\xi_v)\xi_v\ln(t)|W(v)|^2}
\end{align*}
We can immediately see that
$W(v)$ is well-defined, as
$|W(v)|=|\tilde{\gamma}(t,v)|$, which is a constant, and
\[
W(v)=\lim_{\substack{s\rightarrow\infty}}\tilde{\gamma}(e^{2s\pi/(q(\xi_v)\xi_v|W(v)|^2)},v).
\]

\begin{corollary}\label{Asymptotic expansions}
Let $v\in J_\lambda$. Under the assumption $(t,v)\in\mathcal{D}$, we have the asymptotic expansions
\begin{equation}\label{Wdiff}
    \|\gamma(t,v) - W(v)e^{ i q(\xi_v)\xi_v\log t |W(v)|^2}\|_{L^\infty(J_\lambda)} \lesssim \lambda^{-\delta}g(\lambda)t^{-\delta+C^2\eps^2}\eps.
\end{equation}
\begin{equation}\label{Wdiff2}
    \|\gamma(t,v) - W(v)e^{ i q(\xi_v)\xi_v\log t |W(v)|^2}\|_{L^2(J_\lambda)} \lesssim \lambda^{-\delta}(1+\lambda^{-3/2})t^{-\delta+C^2\eps^2}\eps.
\end{equation}
\end{corollary}
\begin{proof}
This is an immediate consequence of Proposition \ref{Asymptotic equation}.
\end{proof}
\begin{proposition}
Under the assumption
\begin{align*}
    \|\varphi_0\|_{X}\lesssim\eps\ll 1,
\end{align*}
the asymptotic profile $W$ defined above satisfies 
\begin{align*}
\|e^{-\frac{v(1+\delta)}{2}}e^{|v|(1/2+\delta/4)}|D_v|^{1-C_1\epsilon^2}W(v)\|_{L_v^2}\lesssim\epsilon.
\end{align*}
Moreover, when $s_0=0$, we also have $\|W(v)\|_{L_v^2}\lesssim\eps$.
\end{proposition}
\begin{proof}
We fix $\lambda$, and let $t\gtrsim\max\{1,\lambda^{-1}\}$. From Corollary \ref{Asymptotic expansions} we know that
\begin{align*}
\|W(v)-e^{- i q(\xi_v)\xi_v\log t |\gamma(t,v)|^2}\gamma(t,v)\|_{L_v^2(J_{\lambda})}\lesssim \lambda^{-\delta}(1+\lambda^{-3/2})t^{-\delta+C^2\eps^2}\eps.
\end{align*}
From the product and chain rules with Lemma \ref{Gamma bounds}, we have 
\begin{align*}
    \left\|\partial_v\left(e^{-i q(\xi_v)\xi_v\log t |\gamma(t,v)|^2}\gamma(t,v)\right)\right\|_{L_v^2(J_{\lambda})}&\lesssim \lambda^{-\delta}(1+\lambda^{-2})\log(t)\eps t^{C^2\eps^2}.
\end{align*}
Putting these together, we find that for $t \gtrsim \max\{1,\lambda^{-1}\}$,
\begin{align*}
W(v)=O_{\dot{H}^1_v(J_\lambda)}(\lambda^{-\delta}(1+\lambda^{-2})\log(t)\eps t^{C^2\eps^2})+O_{L_v^2(J_\lambda)}(\lambda^{-\delta}(1+\lambda^{-3/2})t^{-\delta+C^2\eps^2}\eps).
\end{align*}
By interpolation,  this will imply that
for $C_1$ large enough we have
\begin{align*}
\|W(v)\|_{\dot{H}_v^{1-C_1\epsilon^2}(J_\lambda)}&\lesssim \lambda^{-\delta}(1+\lambda^{-2})\epsilon.
\end{align*}
By dyadic summation over $\lambda\geq 1$ and $\lambda\leq 1$,
\begin{align*}
\|e^{-\frac{v(1+\delta)}{2}}e^{|v|(1/2+\delta/4)}|D_v|^{1-C_1\epsilon^2}W(v)\|_{L_v^2}\lesssim\epsilon.
\end{align*}
The last part immediately follows from the conservation of mass.
\end{proof}

\bibliography{bib-sqg}
\bibliographystyle{plain}

\end{document}